\documentclass[11pt,oneside,english]{amsart}
\usepackage[T1]{fontenc}
\usepackage[latin9]{inputenc}
\usepackage{geometry}
\usepackage{textcomp}
\usepackage{amsthm}
\usepackage{amssymb}
\usepackage{wasysym}
\usepackage{setspace}
\usepackage{esint}
\makeatletter
\numberwithin{equation}{section}
\numberwithin{figure}{section}
\theoremstyle{plain}
\newtheorem{thm}{\protect\theoremname}
  \theoremstyle{definition}
  \newtheorem{defn}[thm]{\protect\definitionname}
  \theoremstyle{plain}
  \newtheorem{prop}[thm]{\protect\propositionname}
  \theoremstyle{plain}
  \newtheorem{lem}[thm]{\protect\lemmaname}
  \theoremstyle{remark}
  \newtheorem*{rem*}{\protect\remarkname}
  \theoremstyle{plain}
  \newtheorem{cor}[thm]{\protect\corollaryname}

\@ifundefined{date}{}{\date{}}
\usepackage{bbm}
\usepackage{dsfont}
\allowdisplaybreaks[4]

\makeatother

\usepackage{babel}
  \providecommand{\corollaryname}{Corollary}
  \providecommand{\definitionname}{Definition}
  \providecommand{\lemmaname}{Lemma}
  \providecommand{\propositionname}{Proposition}
  \providecommand{\remarkname}{Remark}
\providecommand{\theoremname}{Theorem}

\begin{document}
\global\long\def\bbN{\mathbb{N}}

\global\long\def\bbQ{\mathbb{Q}}

\global\long\def\bbR{\mathbb{R}}

\global\long\def\bbZ{\mathbb{Z}}

\global\long\def\bbC{\mathbb{C}}

\global\long\def\eset{\emptyset}

\global\long\def\nto{\nrightarrow}

\global\long\def\re{\mathrm{Re\,}}

\global\long\def\im{\mathrm{Im\,}}

\global\long\def\limti{{\displaystyle \lim_{n\to\infty}}}

\global\long\def\sumnti{{\displaystyle \sum_{n=1}^{\infty}}}

\global\long\def\sumktn{{\displaystyle \sum_{k=1}^{n}}}

\global\long\def\E{{\bf E}}

\global\long\def\ind{\mathbbm{1}}

\global\long\def\O{O}

\global\long\def\Leb{\mathcal{\lambda}}

\global\long\def\calB{\mathcal{B}}

\global\long\def\Re{Re}

\global\long\def\Im{Im}

\global\long\def\limdist{\overset{d}{\longrightarrow}}

\title{Invariance principle for local time by quasi-compactness}

\author{Michael Bromberg\\
School of mathematical sciences, Tel Aviv University, Tel Aviv 69978,
Israel.}
\begin{abstract}
The objective of this paper is to prove a functional weak invariance
principle for a local time of a process of the form $X_{n}=\varphi\circ T^{n}$
where $\left(X,\mathcal{B},T,m\right)$ is a measure preserving system
with a transfer operator acting quasi-compactly on a large enough
Banach space of functions and $\varphi\in L^{2}\left(m\right)$ is
an aperiodic observable.%
\thanks{This research was partially supported by ISF grant 1599/13.%
}%
\thanks{This paper is part of a thesis submitted in partial fulfillment of
the requirements for the degree of Doctor of Philosophy (Mathematics)
at the Tel Aviv University%
} 
\end{abstract}
\maketitle

\section{Introduction}

To introduce the motivation behind the work in this paper, we first
describe the relevant problem and results in the classical case where
$\left(X_{n}\right)$ is a sequence of independent identically distributed
random variables. In this case, setting $S_{n}=\sum_{k=1}^{n}X_{k}$
for $n\geq1$, $S_{0}=0$, the classical invariance principle states
that the sequence $\frac{S_{n}}{\sqrt{n}}$ converges in law to the
Gaussian distribution. Recall that a càdlàg function is a function
that is continuous on the right with finite limits on the left of
every point in its domain of definition. We denote by $D$ the space
of càdlàg functions on $\left[0,1\right]$. Setting $\omega_{n}\left(t\right)=\frac{1}{\sqrt{n}}\sum_{k=0}^{\left[nt\right]}S_{k}$,
$t\in\left[0,1\right]$, where $\left[x\right]$ is the integral value
of $x$, we obtain a sequence of càdlàg functions and a stronger,
functional invariance principle, stating that the random functions
$\omega_{n}\left(\cdot\right)$ converge in law to the Brownian motion
$\omega\left(\cdot\right)$, where $\omega\left(\cdot\right)$ is
uniquely determined by the equality $E\left(\left(\omega\left(1\right)\right)^{2}\right)=E\left(X_{i}^{2}\right)$
(see \cite{Bil}). 

For a general function $f\in D$, the \textit{occupation measure}
of $f$ up to time $1$ is defined by $ $
\begin{equation}
\nu_{f}\left(A\right)=\int_{0}^{1}\ind_{A}\left(f\left(t\right)\right)dt,\,\, A\in B\left(\bbR\right)\label{eq: occupation meas}
\end{equation}
where $B\left(\bbR\right)$ denotes the Borel $\sigma$-field on $\bbR$.
Recall that the occupation measure of the Brownian motion is almost
surely absolutely contiuous with respect to the Lebesgue measure on
$\bbR$\cite{MP}. The (random) density function with respect to the
occupation measure, which we denote by $l\left(x\right)$, is the
local time of the Brownian motion. Thus 
\[
\int\limits _{0}^{1}\ind_{A}\left(\omega\left(t\right)\right)dt=\int\limits _{A}l\left(x\right)dx
\]
for all $A\in\calB\left(\bbR\right)$, at almost every sample point
of the Brownian motion $\omega$. Moreover, $l\left(x\right)$ is
almost surely continuous. 

Local time $l_{n}\left(x\right)$ of the process $\omega_{n}$, (which
we proceed to define in what follows) may be roughly regarded as the
density of the occupation measure $\nu_{\omega_{n}}$, in the sense
that $\nu_{\omega_{n}}\left[a,b\right]-\int_{\left[a,b\right]}l_{n}\left(x\right)dx$
converges to $0$ in law . To define $l_{n}$ we distinguish between
the lattice and non-lattice case, namely between the case when $\left(X_{n}\right)$
is a sequence of random variables taking values in the lattice $\bbZ$,
and when $\left(X_{n}\right)$ is a sequence of random variables taking
values in $\bbR$. 

In the lattice case define 
\[
l_{n}\left(x\right):=n^{-\frac{1}{2}}\#\left\{ k\in\left\{ 1,...,\left[nt\right]\right\} :S_{k}=\left[\sqrt{n}x\right]\right\} 
\]
 where $\left[\cdot\right]$ is the integral value function. Thus,
$l_{n}\left(x\right)$ is the number of arrivals of the random walk
$\left(S_{k}\right)$ at the point $\left(\left[\sqrt{n}x\right]\right)$
up to time $n$, normalized by $\sqrt{n}$. 

In the non-lattice case, let $f:\bbR\rightarrow\bbR$ be continuous
and integrable with integral $1$ and define 
\begin{equation}
l_{n}\left(x\right):=n^{-\frac{1}{2}}\sum_{k=1}^{n}f\left(S_{n}-\sqrt{n}x\right).\label{eq:local time def}
\end{equation}
$ $The invariance principle for local time implies convergence of
$l_{n}\left(x\right)$ to the local time of the Brownian motion $l\left(x\right)$.
More precisely, we say that the invariance principle for local times
holds if the sequence $\left(\omega_{n},l_{n}\right)$ converges in
law to $\left(\omega,l\right)$. 

The invariance principle for local times in the lattice case, under
the assumption of aperiodicity on the random walk was proved in \cite{Bor}.
The invariance principle for the non-lattice case, under assumption
of aperiodicity and an assumption that the characteristic function
of $X_{i}$ is square integrable was proved in \cite{BI}. 

There has been a considerable amount of research invested into generalizing
the invariance principles in the independent case to the more general
settings of various mixing conditions on the processes $\left(X_{n}\right)$,
but to the author's knowlege, no such generalization appeared in literature
for the invariance principle for local time until \cite{BK}, where
the author and Z.Kosloff prove the invariance principle for local
time in the case where $\left(X_{n}\right)$ is a finite state, lattice
valued Markov chain. In \cite{Br}, a further generalization was proved
by the author to the case where $\left(X_{n}\right)$ are of the form
$X_{n}:=\varphi\circ T^{n}$, where $\left(X,\mathcal{C},m,T\right)$
is a Gibbs-Markov system and $\varphi$ is an aperiodic, square integrable
function with values in $\bbZ$. The purpose of this paper is to give
a proof of the invariance principle for local time in the non-lattice
case given that the functional invariance principle holds, under the
setting, where the random variables $\left(X_{n}\right)$ are generated
by a dynamical system $\left(X,\mathcal{C},m,T\right)$ with a quasi-compact
transfer operator.

\subsection{Outline of the remaining sections. }

Section (\ref{sec:Characteristic-function-operator}) describes the
assumptions used in the proof of the main theorem and develops the
basic tools needed for the proof. \ref{sec:Main theorem} describes
the notions of convergence used in this paper and states the main
theorem. Some concrete systems where our assumptions hold, as well
as applications of the main theorem are provided in section \ref{sec:Applications-and-Examples}.
Section \ref{sec:Estimates} provides the probability estimates needed
for the proof. Section \ref{sec:Tightness} proves tigthness of the
local time process, while section \ref{Sec:Identifying the Only Possible Limit Point}
finalizes the proof, by identifying the only posiible limit of the
local time process.

\section{\label{sec:Characteristic-function-operator}Characteristic function
operators and expansion of the main eigenvalue}

Let $\left(X,\mathcal{C},m,T\right)$ be a probability measure preserving
dynamical system. Let $\varphi:X\rightarrow\bbR$ be measurable, and
let 
\begin{equation}
X_{n}:=\varphi\circ T^{n-1},\: S_{n}:=\sum_{k=1}^{n}X_{k},\ n\in\bbN,\ S_{0}=0.\label{eq: Definition of rvs}
\end{equation}
Consider $T$ as an operator on $L^{\infty}\left(m\right)$ defined
by $Tf=f\circ T$. Then the transfer operator, also known as the Frobenius-Perron
operator, $\hat{T}:L^{1}\left(m\right)\rightarrow L^{1}\left(m\right)$
is the pre-dual of $T$, uniquely defined by the equation 
\[
\int f\cdot g\circ T\, d\mu=\int\hat{T}f\cdot g\, d\mu\,\,\forall f\in L^{1},g\in L^{\infty}.
\]
We note for future reference that $\hat{T}$ is a positive operator
in the sense that $\hat{T}f\geq0$ if $f\geq0$ and $\hat{T}$ real
in the sense that if $f$ is a real valued function, then $\hat{T}f$
is real valued. 

The characteristic function operators associated to $\varphi$ is
a family of operators $P\left(t\right):L^{1}\left(m\right)\rightarrow L^{1}\left(m\right)$
defined for all $t\in\bbR$ by 
\[
P\left(t\right)f:=\hat{T}\left(e^{it\varphi}f\right).
\]
Note that at $t=0$ we have the equality $P\left(0\right)=\hat{T}$.
The reason for the name of $P\left(t\right)$ is due to the fact that
$m\left(e^{it\varphi}f\right)=m\left(P\left(t\right)\left(f\right)\right)$
and in particular $m\left(P\left(t\right)\ind\right)$ gives the characteristic
function of $\varphi$. Characteristic function operators may be used
to prove convergence theorems and estimates for $\left(X_{n}\right)$
and $\left(S_{n}\right)$ similarly to the way characteristic functions
are used for the independent case. In particular powers of $P\left(t\right)$
give rise to characteristic functions of $S_{n}$ as shown by the
following equality (which is proved by simple induction using definitions).
\[
m\left(e^{itS_{n}}\right)=m\left(P^{n}\left(t\right)\ind\right).
\]

\begin{defn}
An operator $T$ on a Banach space $\calB$ is called quasi-compact
with $s$ dominating simple eigenvalues if\end{defn}
\begin{enumerate}
\item There exist $T$-invariant spaces $F$ and $H$ such that $F$ is
an $s$ dimensional space and $\calB=F\oplus H$. 
\item $T$ is diagonizable when restricted to $F$, with all eigenvalues
having modulus equal to the spectral radius of $T$ which we denote
by $\rho\left(T\right)$.
\item When restricted to $H$, the spectral radius of $T$ is strictly less
than $\rho\left(T\right)$.
\end{enumerate}
Quasi-compactness of the characteristic function operator acting on
a large enough Banach space of functions, essentially helps in reducing
the behavior of the characteristic functions of $X_{n}$ to a more
familiar i.i.d case. In order to establish an invariance principle
for local time we make the following assumptions.

\subsection{Assumptions.}
\begin{itemize}
\item (A1) There exists a Banach space $\mathcal{B}\subseteq L^{\infty}\left(m\right)$
with norm $\left\Vert \cdot\right\Vert $ satisfying $\left\Vert \cdot\right\Vert _{\infty}\leq C\left\Vert \cdot\right\Vert $
for some $C>0$, such that $\ind\in\mathcal{B}$, and $f\in\calB\implies\bar{f}\in\calB$,
$\left|f\right|\in\calB$.
\item (A2) (Quasi-Compactness) $\hat{T}:\mathcal{B}\rightarrow\mathcal{B}$
is quasi compact with one dominating simple eigenvalue equal to $1$,
and is given by $\hat{T}\left(f\right)=m\left(f\right)\ind+N\left(f\right)$
where the spectral radius of $N$ satisfies $\rho\left(N\right)<1$. 
\item (A3) (Mean zero and finite second moment) $\varphi\in L^{2}\left(m\right)$;
$m\left(\varphi\right)=0$. 
\item (A4) (Continuity) $t\mapsto P\left(t\right)$ is a continuous function
from $\bbR$ to $Hom\left(\mathcal{B},\mathcal{B}\right)$. 
\item (A5) (Differentiability) There exists a neighborhood $I_{0}$ of $0$,
such that for all $t\in I_{0}$, $P\left(t\right):I_{0}\rightarrow Hom\left(\mathcal{B},\mathcal{B}\right)$
is twice continuously differentiable and 
\[
P'\left(0\right)\left(f\right)=\hat{T}\left(i\varphi f\right),\ P''\left(0\right)\left(f\right)=\hat{T}\left(-\varphi^{2}f\right).
\]

\item (A6) (Aperiodicity) The spectral radius of $P\left(t\right)$ as an
element of $Hom\left(\mathcal{B},\mathcal{B}\right)$, satisfies $\rho\left(P\left(t\right)\right)<1$,
$\forall t\neq0$. 
\end{itemize}

\subsection{Remarks.}

(A2) gives quasi-compactness of the characteristic function operator,
with one dominating simple eigenvalue. It follows that the eigenspace
corresponding to the dominating eigenvalue is the space of constant
functions and the projection onto this eigenspace is given by $f\mapsto m\left(f\right)\ind$.
Note that it is a consequence of the definition of the transfer operator
that $1$ is always an eigenvalue of $\hat{T}$, since $\hat{T}\left(\ind\right)=\ind$.
In applications, the requirement for the eigenvalue $1$ to be simple
corresponds to assumption of ergodicity of the system $\left(X,\mathcal{B},m,T\right)$,
while the lack of other eigenvalues of modulus $1$ corresponds to
a weak mixing condition on the system (see section \ref{sec:Applications-and-Examples}
for concrete examples). The condition $\mathcal{B}\subseteq L^{\infty}$
may be replaced by $\mathcal{B}\subseteq L^{p}$, $p\geq1$, with
a similar condition on norms. 

(A4) and (A5) guarantee continuity in $\bbR$ and differentiability
near $0$ of the characteristic function operators. Even though we
assume that $\varphi\in L^{2}\left(m\right)$, (A4) and (A5) do not
follow, since we do not assume that $\varphi\in\mathcal{B}$, or that
$e^{it\varphi}\in\mathcal{B}$ and we make no assumptions about $\mathcal{B}$
being closed under multiplication. Therefore, without (A4) we cannot
even conclude that $P\left(t\right)$ is $\calB$ invariant. In the
Gibbs-Markov case for example (see section \ref{sec:Applications-and-Examples}),
we do not require that $\varphi\in\calB$, but still assumptions (A4)
and (A5) are valid. Note, that the formula for $P'\left(0\right)$
assumes that the derivative of $P'\left(0\right)$ is what one expects
it to be, i.e analogous to the derivative of the characteristic function. 

Finally (A6) corresponds to an assumption of aperiodicity of the function
$f$. The name is derived from references to examples in section \ref{sec:Applications-and-Examples},
where it is shown that this requirement is equivalent to $e^{it\varphi}$
not being cohomologous to a constant. Functions satisfying this last
property are usually called aperiodic. This is a standard assumption
for proving local limit theorems, but is not required for the central
limit theorem. If $\left(X_{n}\right)$ are i.i.d's then aperiodicity
corresponds to the requirement that the modulus of the characteristic
function $E\left(e^{itX_{n}}\right)$ has modulus strictly less than
$1$, for all $t\neq0$. This requirement is satisfied if and only
if the random walk $S_{n}$ does not take values on a lattice in $\bbR$.

\subsection{A perturbation theorem and its implications.}

The proofs of this section follow the methods that first appeared
in \cite{Na} for analytic perturbations (see also \cite{HeH}, \cite{GH}).
We adapt these to our setting. In what follows $C^{m}\left(I,\mathcal{B}\right)$
is used to denote the space of $m$ times continuously differentiable
functions from $I$ to a Banach space $\calB$, $\mathcal{B}^{*}$
is the dual space of $\calB$ and $\left\langle \xi,f\right\rangle $
denotes the action of $\xi\in\calB^{*}$ on $f\in\calB$. The following
proposition is a direct implication of a standard perturbation theorem
(see \cite{HeH}, Theorem III.8).
\begin{prop}
\label{prop:Perturbation theorem}Let assumptions $\left(A1\right)-\left(A5\right)$
be satisfied. Then there exists an open neighborhood of $0$ $I\subseteq I_{0}$,
and functions $\lambda\left(t\right)\in C^{2}\left(I,\bbC\right)$,
$\xi\left(t\right)\in C^{2}\left(I,\mathcal{B}^{*}\right)$, $\eta\left(t\right)\in C^{2}\left(I,\calB\right)$,
$N\left(t\right)\in C^{2}\left(I,Hom\left(\calB,\calB\right)\right)$
such that for $t\in I$, 
\[
P\left(t\right)\eta\left(t\right)=\lambda\left(t\right)\eta\left(t\right),\ Q\left(t\right)^{*}\xi\left(t\right)=\lambda\left(t\right)\xi\left(t\right)\:\left\langle \xi\left(t\right),\eta\left(t\right)\right\rangle =1,
\]
for all $n\geq1$ 
\[
P^{n}\left(t\right)\left(\cdot\right)=\lambda\left(t\right)\left\langle \xi\left(t\right),\cdot\right\rangle \eta\left(t\right)+N^{n}\left(t\right)\left(\cdot\right)
\]
and 
\[
\rho\left(N\left(t\right)\right)<q<\inf_{t\in I}\left|\lambda\left(t\right)\right|,\ \left\Vert N^{n}\left(t\right)\right\Vert \leq Cq^{n}
\]
where $C>0$ and $0<q<1$ are constants. 
\end{prop}
Note that by assumption (A2) and by the fact that $m\left(\ind\right)=1$,
we have $\xi\left(0\right)=m$, $\eta\left(0\right)=\ind$. Defining
$\tilde{\xi}\left(t\right):=\frac{\xi\left(t\right)}{\left\langle \xi\left(0\right),\eta\left(t\right)\right\rangle }$,
$\tilde{\eta}\left(t\right)=\left\langle \xi\left(0\right),\eta\left(t\right)\right\rangle \eta\left(t\right)$
we obtain functionals $\tilde{\xi}\left(t\right)$ and eigenvectors
$\tilde{\eta}\left(t\right)$ satisfying all the conditions of proposition
\ref{prop:Perturbation theorem} in some open neighborhood $I_{1}\subseteq I_{0}$
of $0$, with the extra condition that 
\[
\left\langle \tilde{\xi}\left(0\right),\eta\left(t\right)\right\rangle =m\left(\eta\left(t\right)\right)=1.
\]
$I_{1}$ is chosen so that $\left\langle \xi\left(0\right),\eta\left(t\right)\right\rangle \neq0$
for all $t\in I_{1}$ and is non-empty by continuity and the fact
that $\left\langle \xi\left(0\right),\eta\left(0\right)\right\rangle =1$.
Thus, from now on we may and do assume that $\xi\left(t\right)$,
$\eta\left(t\right)$ of proposition \ref{prop:Perturbation theorem}
satisfy the extra condition 
\begin{equation}
\left\langle \xi\left(0\right),\eta\left(t\right)\right\rangle =m\left(\eta\left(t\right)\right)=1\label{eq: Extra condition}
\end{equation}
for all $t\in I$. Note that this implies that $m\left(\eta'\left(t\right)\right)\equiv0$. 

In what follows, we need more information on the eignevectors and
eigenvalues of $P\left(t\right)$ in $B\left(0,\delta\right)$ which
we summarize in the following lemma. 
\begin{lem}
\label{lem: Perturbation theorem 1}Let $I$,$\lambda\left(t\right)$,
$\xi\left(t\right)$, $\eta\left(t\right)$ be as in proposition \ref{prop:Perturbation theorem}
satisfying (\ref{eq: Extra condition})$ $ and let $\pi\left(t\right)\left(\cdot\right):=\left\langle \xi\left(t\right),\cdot\right\rangle \eta\left(t\right)$.
Then 
\begin{equation}
\lambda\left(t\right)=1-\sigma^{2}t^{2}+o\left(t^{2}\right)\label{eq: Expansion of eigenvalue-1}
\end{equation}
where $\sigma\geq0$, and there exists $\delta>0$ and constants $c,C>0$,
such that for all $t\in\left(-\delta,\delta\right)$
\begin{equation}
\left|\lambda\left(t\right)\right|\leq1-ct^{2}\label{eq: Expansion of eigenvalue}
\end{equation}

\begin{equation}
\left\Vert \pi\left(t\right)f-m\left(f\right)\ind\right\Vert \leq C\left|t\right|\label{eq: Expansion of projection}
\end{equation}
 and $\eta'\left(0\right)$ is a purely imaginary function. $ $ $ $\end{lem}
\begin{proof}
Since $\varphi\left(0\right)=m$ and $\eta\left(0\right)=\ind$, (\ref{eq: Expansion of projection})
follows immediately from Taylor's expansion of $\pi\left(t\right)$
at $0$. To prove the other assertions write 
\begin{eqnarray*}
\lambda\left(t\right) & = & m\left(P\left(t\right)\eta\left(t\right)\right)\\
 & = & m\left(\hat{T}\left(e^{it\varphi}\eta\left(t\right)\right)\right)\\
 & = & m\left(e^{it\varphi}\eta\left(t\right)\right)
\end{eqnarray*}
Now $\left|e^{it\varphi}-1-it\varphi+\frac{t^{2}\varphi^{2}}{2}\right|\leq\left(\varphi t\right)^{2}\cdot\min\left(2,\left|t\varphi\right|\right)$.
$ $Since $\varphi\in L^{2}\left(m\right)$, by dominated convergence
we have $\frac{1}{t^{2}}m\left(\varphi^{2}t^{2}\cdot\min\left(2,\left|t\varphi\right|\cdot\eta\left(t\right)\right)\right)\underset{t\rightarrow0}{\longrightarrow}0$.
This implies,
\[
\lambda\left(t\right)=1+m\left(\left(it\varphi-\frac{t^{2}\varphi^{2}}{2}\right)\eta\left(t\right)\right)+o\left(t^{2}\right).
\]
By Taylor's expansion $\eta\left(t\right)=\ind+\eta'\left(0\right)t+\zeta\left(t\right)$,
where $\left\Vert \zeta\left(t\right)\right\Vert =o\left(\left|t\right|\right)$
and therefore, $\left\Vert \zeta\left(t\right)\right\Vert _{\infty}=o\left(\left|t\right|\right)$.
Thus, since $m\left(\varphi\right)=0$, $\varphi^{2}\in L^{2}\left(m\right)$,
\[
\lambda\left(t\right)=1-m\left(\varphi^{2}\right)\frac{t^{2}}{2}+m\left(\varphi\eta'\left(0\right)\right)it^{2}+o\left(t^{2}\right).
\]
It follows that $\lambda'\left(0\right)=0$. Note that $\varphi\eta'\left(0\right)$
is $m$ integrable because $\eta'\left(0\right)\in\mathcal{B}\subseteq L^{\infty}\left(m\right)$. 

We prove that $\lambda''\left(0\right)\in\bbR$. Write 
\begin{eqnarray*}
\lambda\left(t\right)\eta\left(t\right) & = & P\left(t\right)\eta\left(t\right)\\
 & = & \hat{T}\left(e^{it\varphi}\eta\left(t\right)\right)\\
 & = & \overline{\hat{T}\left(e^{-it\varphi}\overline{\eta\left(t\right)}\right)}.
\end{eqnarray*}
It follows that $P\left(-t\right)\overline{\eta\left(t\right)}=\overline{\lambda\left(t\right)\eta\left(t\right)}$,
and therefore, $\overline{\lambda\left(t\right)}$ is an eigenvalue
of $P\left(-t\right)$. Since by the perturbation theorem $P\left(-t\right)$
has a unique main eigenvalue, and all other eigenvalues are bounded
away from $\inf_{t\in I}\lambda\left(t\right)$, it follows that for
$t\in I$, $\lambda\left(-t\right)=\overline{\lambda\left(t\right)}$.
This implies that $\lambda''\left(0\right)$ is real. $ $ 

We prove that $\lambda''\left(0\right)\leq0$. The following reasoning
is based on the pointwise inequality $\hat{T}^{n}f\leq\hat{T}^{n}\left|f\right|$.
\begin{eqnarray*}
\left|\lambda^{n}\left(t\right)\eta\left(t\right)\right| & = & \left|P^{n}\left(t\right)\eta\left(t\right)\right|=\left|\hat{T}^{n}\left(e^{it\varphi}\eta\left(t\right)\right)\right|\leq\left|\hat{T}^{n}\left(\left|e^{it\varphi}\eta\left(t\right)\right|\right)\right|\\
 & = & m\left(\left|\eta\left(t\right)\right|\right)\ind+N^{n}\left(\left|\eta\left(t\right)\right|\right).
\end{eqnarray*}
Since $\left\Vert N^{n}\right\Vert \longrightarrow0$ and $\mathcal{B}$
is continuously embedded in $L^{\infty}\left(m\right)$, we have 
\[
\left|\lambda^{n}\left(t\right)\right|\left\Vert \eta\left(t\right)\right\Vert _{\infty}\leq m\left(\left|\eta\left(t\right)\right|\right)+\left\Vert N^{n}\left(\left|\eta\left(t\right)\right|\right)\right\Vert _{\infty}\longrightarrow m\left(\left|\eta\left(t\right)\right|\right).
\]
This implies $\left|\lambda\left(t\right)\right|\leq1$. Therefore,
\[
1\geq\left|\lambda\left(t\right)\right|^{2}=\lambda\left(t\right)\overline{\lambda\left(t\right)}=\left(1+\lambda''\left(0\right)t^{2}+o\left(t^{2}\right)\right)^{2},
\]
and it follows that $2\lambda''\left(0\right)\leq0$ if $\left|t\right|$
is small enough. 

We turn to prove that $\eta'\left(0\right)$ is purely imaginary. 

\[
P\left(t\right)\eta\left(t\right)=\lambda\left(t\right)\eta\left(t\right)
\]
implies
\begin{eqnarray*}
P'\left(0\right)\eta\left(0\right)+P\left(0\right)\eta'\left(0\right) & = & \lambda'\left(0\right)\eta\left(0\right)+\lambda\left(0\right)\eta'\left(0\right)\\
 & = & \eta'\left(0\right)
\end{eqnarray*}
where the last equality follows from $\lambda\left(0\right)=1$, $\lambda'\left(0\right)=0$.
Since $\eta\left(0\right)=\ind$ we obtain 
\[
P'\left(0\right)\ind=\left(I-P\left(0\right)\right)\eta'\left(0\right).
\]

Now, by (\ref{eq: Extra condition}) $\left\langle \xi\left(0\right),\eta'\left(0\right)\right\rangle =m\left(\eta'\left(0\right)\right)=0$.
By the perturbation theorem, $P\left(0\right)$ restricted to the
space $\xi\left(0\right)^{\perp}:=\left\{ v\in\calB:\left\langle \xi\left(0\right),v\right\rangle =0\right\} $
satisfies $P\left(0\right)=N\left(0\right)$ and $I-N\left(0\right)$
is invertible on this space with inverse given by $\left(I-N\left(0\right)\right)^{-1}=\sum_{k=0}^{\infty}N^{k}\left(0\right)=\sum_{k=0}^{\infty}P^{k}\left(0\right)_{|\xi\left(0\right)^{\perp}}$.
Since $P'\left(0\right)\ind=\hat{T}\left(i\varphi\right)$ is purely
imaginary and $m\left(\varphi\right)=0$ implies that $P'\left(0\right)\ind\in\xi\left(0\right)^{\perp}$,
we have 
\[
\eta'\left(0\right)=\left(\sum_{k=0}^{\infty}P^{k}\left(0\right)\right)\left(P'\left(0\right)\ind\right)
\]
is purely imaginary as claimed. \end{proof}
\begin{rem*}
Note that only first order differentiability of $P\left(t\right)$
was used in the previous theorem. Nevertheless, we will use derivatives
of second order in \ref{prop: Potential kernel }. \end{rem*}
\begin{lem}
\label{lem: Exponential decay on compacts}Let $K\subseteq\bbR$ be
a compact set such that $0\notin K$. Then under assumptions (A1),
(A4), (A6), there exist constants $C>0$, $0<r<1$ such that for all
$t\in K$ we have $\left\Vert P^{n}\left(t\right)\right\Vert \leq Cr^{n}$. \end{lem}
\begin{proof}
Since the spectral radius is an upper semi-continuous function, by
(A6) there exists $r<1$, such that $\sup_{t\in K}\rho\left(P\left(t\right)\right)<r<1$.
It follows from Gelfand's formula for the spectral radius of an operator
that $r\geq\rho\left(P\left(t\right)\right)=\lim_{n\rightarrow\infty}\sqrt[n]{\left\Vert P^{n}\left(t\right)\right\Vert }$.
Thus for $\epsilon$ such that $r+\epsilon<1$, we have $\left\Vert P^{n}\left(t\right)\right\Vert \leq\left(r+\epsilon\right)^{n}$
if $n$ is large enough. The conclusion of the lemma follows from
this. 
\end{proof}

\section{\label{sec:Main theorem}Statement of the main theorem}

Recall that for a sequence of random variables $\left(X_{n}\right)$
taking values in a complete and separable metric space $\left(M,d\right)$
converges in distribution (or in law) to $X$ if for every continuous
and bounded $f:M\rightarrow\bbR$ 
\[
E\left(f\left(X_{n}\right)\right)\longrightarrow E\left(f\left(X\right)\right).
\]
In this case we denote $X_{n}\limdist X$. 

Let $\left(X,\mathcal{C},m,T\right)$ a probability preserving system
and $\varphi:X\rightarrow\bbR$ a measurable function. Assume that
assumptions (A1)-(A6) are satisfied and let $X_{n}$, $S_{n}$ be
defined by (\ref{eq: Definition of rvs}). By proposition \ref{prop:Perturbation theorem}
and lemma \ref{lem: Perturbation theorem 1} 
\begin{eqnarray*}
m\left(e^{it\frac{S_{n}}{\sqrt{n}}}\right) & = & m\left(P^{n}\left(\frac{t}{\sqrt{n}}\right)\ind\right)\\
 & = & \lambda^{n}\left(\frac{t}{\sqrt{n}}\right)+m\left(N^{n}\left(\frac{t}{\sqrt{n}}\right)\ind\right)\\
 & = & \left(1-\sigma\frac{t^{2}}{n}+o\left(\frac{t^{2}}{n}\right)\right)^{n}+m\left(N^{n}\left(t\right)\ind\right)
\end{eqnarray*}
Therefore, $\lim_{n\rightarrow\infty}m\left(e^{it\frac{S_{n}}{\sqrt{n}}}\right)=e^{-\sigma t^{2}}$.
Thus, setting $a^{2}=\frac{\sigma}{2}$ it follows that $\frac{S_{n}}{\sqrt{n}}$
converges in distribution to the Gaussian distribution with mean $0$
and variance $a^{2}$ ((A6) is not used for the central limit theorem).
Note that the limit is degenerate if and only $\lambda''\left(0\right)=\sigma=0$. 

Let 
\[
\omega_{n}\left(t\right):=\frac{1}{\sqrt{n}}\sum_{k=0}^{\left[nt\right]}S_{k},\, t\in\left[0,1\right].
\]
It is easily seen by definition that $\omega_{n}\left(t\right)$ is
a càdlàg function on $\left[0,1\right]$ (continuous from the right
with limits from the left). As stated in the introduction, we denote
by $D$ the Skorokhod space of càdlàg functions on $\left[0,1\right]$.
Recall that endowed with the Skorokhod metric which we denote by $d_{J}\left(\cdot,\cdot\right)$
(see \cite{Bil}) $D$ is complete and separable. We denote by $\omega\left(t\right)$
the Brownian motion on $\left[0,1\right]$, uniquely defined by the
equalities $E\left(\omega\left(t\right)\right)=0,$ $\forall t\in\left[0,1\right]$,
$E\left(\omega\left(1\right)^{2}\right)=\frac{\sigma}{2}$ (here $E\left(\cdot\right)$
denotes expectation with respect to the Wiener measure). It can be
easily seen by arguments similar to the above that for $0\leq t_{1}\leq...\leq t_{k}\leq1$
we have
\begin{align}
\Bigl(\omega_{n}\left(t_{1}\right),\omega_{n}\left(t_{2}\right)- & \omega_{n}\left(t_{1}\right),...,\omega_{n}\left(t_{k}\right)-\omega\left(t_{k-1}\right)\Bigr)\nonumber \\
 & \limdist\left(\omega\left(t_{1}\right),\omega\left(t_{2}\right)-\omega\left(t_{1}\right)...,\omega\left(t_{k}\right)-\omega\left(t_{k-1}\right)\right)\label{eq: Finite dim to Brownian motion}
\end{align}
 The functional central limit theorem (or the functional invariance
principle) is a statement that $\omega_{n}\limdist\omega$, where
convergence takes place in the Skorokhod space $D$. The functional
invariance principle does not seem to follow from assumptions (A1)-(A5).
To prove it, in addition to (\ref{eq: Finite dim to Brownian motion})
one has to show that the sequence $\omega_{n}$ is tight in $D$ (for
details on tightness see section \ref{sec:Estimates} or \cite{Bil}).
If in addition to (A1)-(A5) one assumes that $\varphi\in\mathcal{B}$
and $\varphi^{2}\in\mathcal{B}$, one can show that for $r<s<t$,
\[
m\left(\left(S_{\left[nt\right]}-S_{\left[ns\right]}\right)^{2}\left(S_{\left[ns\right]}-S_{\left[nr\right]}\right)^{2}\right)\leq C\left|t-r\right|^{2}.
\]
which implies tightness (see \cite{Bil}). Instead of assuming these
extra conditions, which are not required for our proof of the invariance
principle of local time we assume that the functional invariance principle
holds. In section \ref{sec:Applications-and-Examples} we provide
references for the functional invariance principle in concrete cases.
Thus we add the extra assumption:
\begin{itemize}
\item (A7) $\omega_{n}$ converges in law to $\omega$ in the space $D$,
where $\omega$ is the Brownian motion satisfying $E\left(\omega\left(t\right)\right)=0$
$\forall t\in\left[0,1\right]$, $E\left(\omega\left(1\right)^{2}\right)>0$. 
\end{itemize}
Note that by (\ref{eq: Finite dim to Brownian motion}), $\omega_{n}$
cannot converge to anything else except $\omega$, and the requirement
that $E\left(\omega\left(1\right)^{2}\right)>0$ is equivalent to
stating that the limit of $\omega_{n}$ is non-degenerate. This in
turn happens if and only if $\lambda''\left(0\right)=\sigma>0$. 

Let $f:\bbR\rightarrow\bbR$ be a smooth, integrable, symmetric function
with compactly supported Fourier transform. In what follows $\hat{f}$
denotes the Fourier transform of $f$. Let $l_{n}$ be the local time
of $\omega_{n}$ defined by 
\[
l_{n}:=l_{n}\left(x\right):=n^{-\frac{1}{2}}\sum_{k=1}^{n}f\left(S_{n}-\sqrt{n}x\right).
\]
By continuity of $f$ it follows that $m$-almost surely $l_{n}$
takes values in the space of continuous functions on $\left(-\infty,\infty\right)$
denoted by $C$. Endow $C$ with the topology of uniform convergence
on compact sets. With respect to the metric $d\left(f,g\right)=\frac{1}{2^{n}}\sup_{\left[-n,n\right]}\left\Vert f-g\right\Vert _{\infty}$
, $C$ is separable and complete. Let $l$ be the local time of the
Brownian motion $\omega$. We are now in the position to state the
main theorem:
\begin{thm}
\label{thm:Main thm}Let $\left(X,\mathcal{C},m,T\right)$ be a probability
preserving system and $\varphi:X\rightarrow\bbR$ be such that assumptions
(A1)-(A7) hold. Then the sequence $\left(\omega_{n},l_{n}\right)$
converges in law to $\left(\omega,\int_{\bbR}f\left(x\right)dx\cdot l\right)$
in the space $D\times C$. \end{thm}
\begin{rem*}
Instead of assuming that the function $f$ has compactly supported
Fourier transform, we may assume, in addition to (A6) that $\limsup_{t\rightarrow\infty}\rho\left(P\left(t\right)\right)<1$.
This is the so called Cramer's condition on the function $\varphi$.
It allows to extend the statement of lemma \ref{lem: Exponential decay on compacts}
to non-compact intervals that are bounded away from $0$, which allows
to carry out the estimates in section \ref{sec:Estimates} without
the assumption on $f$ having compactly carried Fourier transform.
In this case the theorem would be valid for any symmetric, integrable
function $f$, and in particular for functions of the form $f=\ind_{\left(-a,a\right)}$,
where $a>0$. 
\end{rem*}
\global\long\def\newmacroname{}

\section{\label{sec:Applications-and-Examples}Applications and Examples}

The theorem is applicable for systems where one can show that the
transfer operator acts quasi-compactly on a large enough Banach space.
We briefly describe two concrete example of subshifts of finite type
and their generalization to a non-compact space via Gibbs-Markov maps
and refer the reader to \cite{ADSZ,HeH,LY,Yo} for other examples.

\subsection{Subshifts of finite type. }

We refer the reader to \cite{Bow} as a basic reference for subshifts
of finite type. Denote by $\bbN_{*}$ the set $\bbN\bigcup\left\{ 0\right\} $.
For $d\in\bbN$, let $\mathcal{S}=\left\{ 1,...,d\right\} $. Endow
$S^{\bbN_{*}}$ with the (compact) metric $d_{\theta}\left(x,y\right):=\theta^{t\left(x,y\right)}$
where $0<\theta<1$, and $t\left(x,y\right)=\min\left\{ n:\: x_{n}\neq y_{n}\right\} $,
and let $\sigma:\mathcal{S}^{\bbN}\rightarrow\mathcal{S}^{\bbN}$
be the left shift operator, defined by $\left(\sigma x\right)_{n}=x_{n+1}$.
Let $A:\mathcal{S}\times\mathcal{S}\rightarrow\left\{ 0,1\right\} $
be an irreducible, aperiodic matrix, i.e there exists some integer
$n_{0}$, such that all entries of $A^{n_{0}}$ are strictly positive
and let $\Sigma_{+}:=\left\{ x\in\mathcal{S}^{\bbN_{*}}:\, A\left(x_{i},x_{i+1}\right)=1\,\forall i\in\bbN_{*}\right\} $.
Then $\Sigma_{+}$ is a closed, shift invariant subspace of $\mathcal{S}^{\bbN_{*}}$.
Let $C\left(\Sigma_{+}\right)$ be the Banach space of all continuous
complex valued functions on $\Sigma_{+}$ endowed with the supremum
norm $\left\Vert \cdot\right\Vert _{\infty}$. Define $F_{\theta}\left(\Sigma_{+}\right)\subseteq C\left(\Sigma_{+}\right)$
to be the set of all Lipchitz continuous functions on $\Sigma_{+}$.
Endowed with the norm $\left\Vert \cdot\right\Vert =\left\Vert \cdot\right\Vert _{\infty}+\left\Vert \cdot\right\Vert _{Lip}$,
where $\left\Vert f\right\Vert _{Lip}=\sup_{x,y\in\Sigma_{+}}\frac{\left|f\left(x\right)-f\left(y\right)\right|}{d_{\theta}\left(x,y\right)}$,
$F_{\theta}\left(\Sigma_{+}\right)$ becomes a Banach space. For a
function $\phi\in F_{\theta}\left(\Sigma_{+}\right)$, there exists
a unique, $\sigma$-invariant Borel measure $m_{\phi}$, called the
Gibbs measure with respect to $\phi$ satisfying 
\[
c_{1}\leq\frac{m_{\phi}\left\{ y\in\Sigma_{+},\, x_{i}=y_{i},\, i=1,...,n\right\} }{\exp\left(-Pn+\sum_{i=0}^{n-1}\phi\left(\sigma^{i}x\right)\right)}\leq c_{2}
\]
for some constants $c_{1}>0$, $c_{2}>0$, $P$ and all $x\in\Sigma_{+}$,
$n\geq0$. 

It is clear that the Banach space $\mathcal{B}=F_{\theta}\left(\Sigma_{+}\right)$
satisfies (A1). It is a consequence of the Ruelle-Perron-Frobenius
theorem combined with the assumption of irreducibility and aperiodicity
of the matrix $A$ that the transfer operator $\hat{\sigma}$ of the
system $\left(\Sigma_{+},\mathcal{C},m_{\phi},\sigma\right)$ satisfies
(A2) (here $\mathcal{C}$ is the Borel $\sigma$-algebra on $\Sigma_{+}$).
Let $\varphi\in\mathcal{B}$ such that $m\left(\varphi\right)=0$.
Note that $\varphi\in\mathcal{B}$ implies that $\varphi$ is bounded
and therefore, has finite moments of any order. It is easy to see
that $\mathcal{B}$ is closed under multiplication. This implies that
the characteristic function operator has continuous derivatives of
any order with derivatives given by $P^{\left(k\right)}\left(t\right)\left(f\right)=\hat{\sigma}\left(i^{k}\varphi^{k}e^{it\varphi}f\right)$.
Thus, assumptions (A1)-(A5) are satisfied for $\varphi\in\mathcal{B}$. 

Assumption (A6) is equivalent to the following: for any $t\in\bbR\setminus\left\{ 0\right\} $,
$e^{it\varphi}$ is not $\sigma$-cohomologous to a constant, i.e.
the only solution to the equation 
\begin{equation}
e^{it\varphi}=\frac{\lambda f\circ\sigma}{f},\ \lambda\in\mathbb{T},\ f:\Sigma_{+}\rightarrow\mathbb{T},\ \mathbb{T}=\left\{ z\in\bbC:z=\left|1\right|\right\} \label{eq:aperiodicity}
\end{equation}
is $\lambda=1$, $f\equiv1$. A function $\varphi$ satisfying this
assumption is called aperiodic. 

Finally, (A7) fails if and only if $\varphi$ is a coboundary, i.e.
there exists $g:\Sigma_{+}\rightarrow\bbR$ measurable, such that
$\varphi=g\circ\sigma-g$. For the proof of the functional invariance
principle refer to \cite{BS}. 

As an application in ergodic theory of the invariance principle for
local time we refer the reader to \cite{Au} where it is used to prove
that the entropy of the scenery is an invariant for random walks in
random scenery processes with a subshift of finite type at the base.

\subsection{Gibbs-Markov maps. }

We refer the reader to \cite{AD} as a basic reference for Gibbs-Markov
maps. Let $\left(X,\mathcal{\mathcal{C}},m,T\right)$ be a probability
preserving transformation of a standard probability space. $T$ is
a Markov map, if there exists a countable partition $\alpha$ of $X$
such that $T\left(\alpha\right)\subseteq\sigma\left(\alpha\right)$
($\mod m$), $T$ when restrict to each element of the partition $\alpha$
is invertible and $\bigcup_{n=0}^{\infty}\left\{ T^{-n}\alpha\right\} $
generates $\mathcal{C}$ (here $\sigma\left(\alpha\right)$ is the
$\sigma$-algebra generated by $\alpha$). Write $\alpha=\left\{ a_{s}:s\in\mathcal{S}\right\} $
and endow $S^{\bbN}$ with the metric $d_{\theta}\left(x,y\right):=\theta^{t\left(x,y\right)}$
where $0<\theta<1$, and $t\left(x,y\right)=\min\left\{ n:x_{n}\neq y_{n}\right\} $.
Set $\Sigma=\left\{ s\in S^{\bbN_{*}}:\mu\left(\bigcap_{k=1}^{n}T^{-k+1}a_{s_{k}}\right)>0\,\forall n\geq1\right\} $.
Then $\Sigma$ is a closed, shift invariant subset of $\mathcal{S}^{\bbN}$
and the system $\left(X,\mathcal{C},m,T\right)$ is conjugate to $\left(\Sigma,\calB\left(\Sigma\right),m,\sigma\right)$
by the map $\left\{ \varphi\left(s_{1},s_{2},...\right)\right\} :=\bigcap_{k=0}^{\infty}T^{-k}a_{s_{k}}$,
where $\sigma$ is the left shift, and $m:=m\circ\varphi$. Thus,
we may assume that $X=\Sigma$, $\mathcal{C}=\calB\left(\Sigma\right)$,
$T=\sigma$ and $\alpha=\left\{ \left[s\right]:s\in\mathcal{S}\right\} $,
where $\left[s_{1},s_{2},...,s_{n}\right]$ denotes the cylinder $\left\{ x\in\mathcal{S}^{\bbN}:x_{i}=s_{i}\,\forall i\leq n\right\} $.
A Markov map $\left(X,\calB,m,T,\alpha\right)$ is Gibbs-Markov if
two additional properties are satisfied:
\begin{itemize}
\item (Big image property) $\inf_{a\in\alpha}m\left(Ta\right)>0$. 
\item (Bounded distortion) For $a\in\alpha$, denote by $f\left(x\right)$
the jacobian of the map $T^{-1}:Ta\rightarrow a$, i.e $f\left(x\right)=\left(\frac{dm\circ T_{|a}}{dm_{|a}}\left(x\right)\right)^{-1}$.
There exists $M>0$ such that for all $a\in\alpha$, and almost every
$x,y\in Ta$, 
\[
\left|1-\frac{f\left(x\right)}{f\left(y\right)}\right|<Md\left(x,y\right).
\]
 
\end{itemize}
A function $f:X\rightarrow\bbR$ is Lipchitz continuous on a set $A\subseteq X$
if 
\[
D_{A}\left(f\right):=\sup_{x,y\in A}\frac{f\left(x\right)-f\left(y\right)}{d\left(x,y\right)}<\infty.
\]
For a partition $ $$\tau$ of $X$ let $D_{\tau}\left(f\right):=\sup_{a\in\tau}D_{a}f$
and let $Lip_{q,\tau}$ be the space 
\[
\left\{ f\in L^{q}\left(m\right):D_{\tau}\left(f\right)<\infty\right\} 
\]
$Lip_{q,\tau}$ is a Banach space with respect to the norm $\left\Vert f\right\Vert :=\left\Vert f\right\Vert _{q}+D_{\rho}\left(f\right)$.
We consider the space $\mathcal{B}=Lip_{\infty,\beta}$ where $\beta=T\alpha$.
Clearly $\mathcal{B}$ satisfies (A1). It is shown in \cite{AD} that
if $T$ is mixing then the transfer operator $\hat{T}$ satisfies
(A2) and (A4) for $\varphi\in Lip_{2,\alpha}$, $m\left(\varphi\right)=0$.
To show that (A5) holds note that 
\begin{eqnarray*}
P\left(t\right)f-P\left(0\right)f & = & \hat{T}\left(\left(e^{it\varphi}-1\right)f\right)\\
 & = & \hat{T}\left(-i\varphi tf+\frac{\varphi^{2}t^{2}}{2}f+\varphi^{2}t^{2}\min\left(\left|\varphi t\right|,2\right)\cdot f\right)\\
 & = & \hat{T}\left(i\varphi tf\right)+\hat{T}\left(\varphi^{2}f\right).
\end{eqnarray*}
As $\left(e^{it\varphi}-1-it\varphi+\frac{\varphi^{2}t^{2}}{2}\right)\leq\varphi^{2}t^{2}\min\left(\left|\varphi t\right|,2\right)$
we have that 
\[
\hat{T}\left(\left(e^{it\varphi}-1\right)f\right)=\hat{T}\left(-i\varphi tf\right)+\hat{T}\left(\varphi^{2}t^{2}f\right)+o\left(t^{2}\right).
\]
Proposition 1.4 of \cite{AD} shows that $\hat{T}:Lip_{1,\beta}\rightarrow\mathcal{B}$.
Therefore, $\left\Vert \hat{T}\left(-i\varphi tf\right)\right\Vert <\infty$,
$\left\Vert \hat{T}\left(\varphi^{2}f\right)\right\Vert <\infty$
and assumption (A5) follows from this. As in the case of subshifts
of finite type (A6) holds if and only the only solutions to (\ref{eq:aperiodicity})
are $\lambda=1$, $f\equiv1$. The functional invariance principle
for Gibbs Markov maps follows from a stronger, almost sure invariance
principle proved for example in \cite{Gou} for observables in $L^{p}$
with $p>2$.

\section{Estimates\label{sec:Estimates}}

In this section we obtain the main estimates, used in the proof of
theorem \ref{thm:Main thm}. Henceforth we assume that assumptions
(A1)-(A7) hold and use the notation introduced in section \ref{sec:Characteristic-function-operator}.
In proofs throughout this section, we use the notation $a\apprle b$
to mean that there exists a constant $C$ such that $a\leq Cb$. 
\begin{prop}
\label{prop:LLT}There exists a constant $C$ such that for all $n\in\bbN$,
$m\left(f\left(S_{n}-x\right)\right)\leq\frac{C}{\sqrt{n}}$, $\int_{\bbR}\left|\lambda^{n}\left(t\right)\right|\leq\frac{C}{\sqrt{n}}$. \end{prop}
\begin{proof}
Let $\delta$ be as in lemma \ref{lem: Perturbation theorem 1} and
set $C_{\delta}:=\left(-\delta,\delta\right)$, $\bar{C}_{\delta}=\bbR\setminus\left(-\delta,\delta\right)$.
By inversion formula for Fourier transform and by definition of the
characteristic function operator, we have
\begin{eqnarray*}
m\left(f\left(S_{n}-x\right)\right) & = & m\left(\int_{\bbR}\hat{f}\left(t\right)e^{it\left(S_{n}-x\right)}dt\right)\\
 & = & \int_{\bbR}\hat{f}\left(t\right)P^{n}\left(t\right)\left(\ind\right)e^{-itx}dt\\
 & \leq & \int_{C_{\delta}}\left|\hat{f}\left(t\right)\right|\left\Vert P^{n}\left(t\right)\right\Vert dt+\int_{\bar{C}_{\delta}}\left|\hat{f}\left(t\right)\right|\left\Vert P^{n}\left(t\right)\right\Vert dt.
\end{eqnarray*}
Since $\hat{f}\left(t\right)$ has compact support, by lemma \ref{lem: Exponential decay on compacts},
the second term exponentially tends to $0$. We estimate the first
term. By the expansion of the characteristic function operator, 
\begin{eqnarray*}
\int\limits _{C_{\delta}}\left|\hat{f}\left(t\right)\right|\left\Vert P^{n}\left(t\right)\right\Vert dt & \leq & \int\limits _{C_{\delta}}\left|\hat{f}\left(t\right)\right|\left(\left|\lambda^{n}\left(t\right)\right|\left\Vert \pi\left(t\right)\right\Vert +\left\Vert N^{n}\left(t\right)\right\Vert \right)dt\\
 & \apprle & \int\limits _{C_{\delta}}\left(1-ct^{2}\right)^{n}dt+\int\limits _{C_{\delta}}\left\Vert N^{n}\left(t\right)\right\Vert dt.
\end{eqnarray*}
Since $\left\Vert N^{n}\left(t\right)\right\Vert $ exponentially
tends to $0$, the assertion is satisfied for the second term. 

Changing variables $x=\frac{t}{\sqrt{n}}$ in the first term we get
\begin{eqnarray*}
\int\limits _{C_{\delta}}\left(1-ct^{2}\right)^{n}dt & = & \frac{1}{\sqrt{n}}\int\limits _{-\sqrt{n}\delta}^{\sqrt{n}\delta}\left(1-c\frac{x^{2}}{n}\right)^{n}dx\\
 & \leq & \frac{1}{\sqrt{n}}\int\limits _{-\infty}^{\infty}e^{-cx^{2}}dx\apprle\frac{1}{\sqrt{n}}.
\end{eqnarray*}
Thus $m\left(f\left(S_{n}-x\right)\right)\leq\frac{C}{\sqrt{n}}$
for some $C>0$. The proof that $m\left(\left|\lambda^{n}\left(t\right)\right|\right)\leq\frac{C}{\sqrt{n}}$
is contained in the above proof.\end{proof}
\begin{prop}
\label{prop: Potential kernel }(Potential Kernel Estimate) There
exists a constant $C>0$, such that for all $y\in\bbR$, 
\[
\sum_{n=1}^{\infty}\left|m\left(f\left(S_{n}\right)-f\left(S_{n}+y\right)\right)\right|\leq C\left|y\right|.
\]
\end{prop}
\begin{proof}
By the inversion formula for Fourier transform,
\begin{align*}
\left|m\left(f\left(S_{n}\right)-f\left(S_{n}+y\right)\right)\right| & =\left|m\left(\Re\int_{\bbR}\hat{f}\left(t\right)e^{itS_{n}}\left(1-e^{ity}\right)dt\right)\right|\\
 & =\left|\Re\int_{\bbR}\hat{f}\left(t\right)m\left(P^{n}\left(t\right)\left(\ind\right)\right)\left(1-e^{ity}\right)dt\right|
\end{align*}
where the first equality follows since the left side is real and the
second inequality is valid by Fubini's theorem. By proposition \ref{prop:Perturbation theorem}
and lemma \ref{lem: Perturbation theorem 1} there exist a $\delta>0$
such that for every $t\in\left(-\delta,\delta\right)$, 
\[
P\left(t\right)\left(\cdot\right)=\lambda\left(t\right)\pi\left(t\right)\left(\cdot\right)+N\left(t\right)\left(\cdot\right)
\]
where $\left|\lambda\left(t\right)\right|\leq1-ct^{2}$ for some positive
constant $c$, the spectral radius of $N\left(t\right)$ satisfies
$\rho\left(N\left(t\right)\right)\leq q<1$ for all $t\in\left(-\delta,\delta\right)$,
and $\pi\left(t\right)=m\ind+\zeta\left(t\right)$ with $\left\Vert \zeta\left(t\right)\right\Vert \leq Ct$
for some $C\geq0$. Write $C_{\delta}=\left(-\delta,\delta\right)$
and $\bar{C}_{\delta}=\bbR\setminus\left(-\delta,\delta\right)$.
Then 
\begin{eqnarray}
\left|\Re\int_{\bbR}\hat{f}\left(t\right)m\left(P^{n}\left(t\right)\ind\right)\left(1-e^{ity}\right)dt\right| & \leq & \left|\Re\int_{C_{\delta}}\hat{f}\left(t\right)m\left(P^{n}\left(t\right)\ind\right)\left(1-e^{ity}\right)dt\right|\label{eq:Estimation 1 - potential kerenel}\\
 &  & +\left|\int_{\bar{C}_{\delta}}\hat{f}\left(t\right)m\left(P^{n}\left(t\right)\ind\right)\left|1-e^{ity}\right|dt\right|\nonumber 
\end{eqnarray}
Since the support of $\hat{f}$ is compact by lemma \ref{lem: Exponential decay on compacts}
there exists $0<r<1$, such that $\left\Vert P^{n}\left(t\right)\right\Vert \apprle r^{n}$
on $\bar{C}_{\delta}$. This, together with $\left|1-e^{ity}\right|\apprle\left|y\right|$
implies 
\[
\sum_{n=1}^{\infty}\left|\int_{\bar{C}_{\delta}}\hat{f}\left(t\right)m\left(P^{n}\left(t\right)\ind\right)\left|1-e^{ity}\right|dt\right|\apprle\frac{\left|y\right|}{1-r}\apprle\left|y\right|.
\]
 To bound the right hand side of (\ref{eq:Estimation 1 - potential kerenel})
use the expansion of the characteristic function operator to get

\begin{eqnarray}
\left|\Re\int_{C_{\delta}}\hat{f}\left(t\right)m\left(P^{n}\left(t\right)\ind\right)\left(1-e^{ity}\right)dt\right| & \leq & \left|\Re\int_{C_{\delta}}\hat{f}\left(t\right)\lambda^{n}\left(t\right)m\left(\pi\left(t\right)\ind\right)\left(1-e^{ity}\right)dt\right|\nonumber \\
 &  & +\int_{C_{\delta}}\left|\hat{f}\left(t\right)\right|\left\Vert N^{n}\left(t\right)\right\Vert \left|1-e^{ity}\right|dt.\label{eq: potential kernel 3}
\end{eqnarray}
Since $\rho\left(N\left(t\right)\right)\leq q<1$, and $\left|1-e^{ity}\right|\apprle\left|y\right|$,
\[
\sum_{n=1}^{\infty}\int_{C_{\delta}}\left|\hat{f}\left(t\right)\right|\left\Vert N\left(t\right)\right\Vert ^{n}\left|1-e^{ity}\right|dt\apprle\left|y\right|.
\]
 We turn to analyze the first term on the right hand side of the inequality
(\ref{eq: potential kernel 3}). Since $\hat{f}\left(t\right)$ is
real valued because $f$ is symmetric,

\begin{align}
\biggl|\Re\int_{C_{\delta}}\hat{f}\left(t\right)\lambda^{n}\left(t\right) & m\left(\pi\left(t\right)\ind\right)\cdot\left(1-e^{ity}\right)dt\biggr|\nonumber \\
= & \left|\int_{C_{\delta}}\hat{f}\left(t\right)\Re\left(\lambda^{n}\left(t\right)m\left(\pi\left(t\right)\ind\right)\right)\Re\left(1-e^{ity}\right)dt\right|\nonumber \\
 & +\left|\int_{C_{\delta}}\hat{f}\left(t\right)\Im\left(\lambda^{n}\left(t\right)m\left(\pi\left(t\right)\ind\right)\right)\Im\left(1-e^{ity}\right)dt\right|\label{eq: Potential kernel to estimate}
\end{align}
Since $\left|\Re\lambda^{n}\left(t\right)\right|\leq\left|\lambda^{n}\left(t\right)\right|\leq1-ct^{2}$
, and $\left\Vert \pi\left(t\right)\right\Vert \apprle\left|t\right|$,
\begin{align}
\sum_{n=1}^{\infty}\biggl|\int_{C_{\delta}}\hat{f}\left(t\right)\Re\left(\lambda^{n}\left(t\right)m\left(\pi\left(t\right)\ind\right)\right) & \Re\left(1-e^{ity}\right)dt\biggr|\label{eq: Estimation 2- Potential kernel}\\
 & \leq\sum_{n=1}^{\infty}\int\limits _{C_{\delta}}\left(1-ct^{2}\right)^{n}\left|1-\cos ty\right|dt\nonumber \\
 & =\int\limits _{C_{\delta}}\frac{1}{ct^{2}}\left|1-\cos ty\right|dt\nonumber \\
 & =\frac{2}{c}\int\limits _{0}^{\left|\frac{1}{y}\right|}\frac{1}{t^{2}}\left|1-\cos ty\right|dt+\frac{2}{c}\int\limits _{\left|\frac{1}{y}\right|}^{\delta}\frac{1}{t^{2}}\left|1-\cos ty\right|dt\nonumber 
\end{align}
Since $\left|1-\cos ty\right|\leq\left|ty\right|^{2}$ we have 
\begin{equation}
\int\limits _{0}^{\left|\frac{1}{y}\right|}\frac{1}{t^{2}}\left|1-\cos ty\right|dt\leq\left|y\right|.\label{eq: Estimation 56 - potential kernel}
\end{equation}
Now if $\left|\frac{1}{y}\right|\geq\delta$, then $\frac{2}{c}\int_{\left|\frac{1}{y}\right|}^{\delta}\frac{1}{t^{2}}\left|1-\cos ty\right|dt\leq0$
and therefore, by (\ref{eq: Estimation 56 - potential kernel}) and
(\ref{eq: Estimation 2- Potential kernel})
\[
\sum_{n=1}^{\infty}\left|\int_{C_{\delta}}\hat{f}\left(t\right)\Re\left(\lambda^{n}\left(t\right)m\left(\pi\left(t\right)\ind\right)\right)\Re\left(1-e^{ity}\right)dt\right|\apprle\left|y\right|.
\]
On the other hand, if $\left|\frac{1}{y}\right|<\delta$, then 
\[
\int_{\left|\frac{1}{y}\right|}^{\delta}\frac{1}{t^{2}}\left|1-\cos ty\right|dt\leq\int_{\left|\frac{1}{y}\right|}^{\infty}\frac{2}{t^{2}}dt=2\left|y\right|.
\]
Combining this with (\ref{eq: Estimation 56 - potential kernel})
and (\ref{eq: Estimation 2- Potential kernel}) again yields 
\[
\sum_{n=1}^{\infty}\left|\int_{C_{\delta}}\hat{f}\left(t\right)\Re\left(\lambda^{n}\left(t\right)m\left(\pi\left(t\right)\ind\right)\right)\Re\left(1-e^{ity}\right)dt\right|\apprle\left|y\right|.
\]
We estimate the sum over the second term in (\ref{eq: Potential kernel to estimate}).
Using $\pi\left(t\right)=m+\zeta\left(t\right)$, $\left\Vert \zeta\left(t\right)\right\Vert \apprle\left|t\right|$
we obtain
\begin{equation}
\left|\int_{C_{\delta}}\hat{f}\left(t\right)\left(\Im\lambda^{n}\left(t\right)m\left(\pi\left(t\right)\ind\right)\right)\left(\sin ty\right)dt\right|\apprle\int_{C_{\delta}}\left|\Im\lambda^{n}\left(t\right)\right|\left|\sin ty\right|dt+\mbox{\ensuremath{\int}}_{C_{\delta}}\left|t\cdot\lambda^{n}\left(t\right)\right|\left(\sin ty\right)dt.\label{eq:Potential kernel to estimate 2}
\end{equation}
Using $\left|\lambda\left(t\right)\right|\leq1-ct^{2}$ we can estimate
the second term on the right hand side of the above inequality. 
\[
\sum_{n=1}^{\infty}\mbox{\ensuremath{\int}}_{C_{\delta}}\left|t\cdot\lambda^{n}\left(t\right)\right|\left|\sin ty\right|dt\leq\int_{C_{\delta}}\frac{1}{ct}\left|\sin ty\right|dt\apprle\left|y\right|.
\]
The estimation of the first term on the right hand side of \ref{eq:Potential kernel to estimate 2}
will take up the rest of the proof. 

We first note that $\left|\Im\lambda^{n}\left(t\right)\right|\leq n\left|\lambda^{n-1}\left(t\right)\right|\left|\Im\lambda\left(t\right)\right|$.
Then 
\begin{eqnarray*}
\left|\Im\lambda\left(t\right)\right| & = & \left|m\left(\Im P\left(t\right)\eta\left(t\right)\right)\right|\\
 & \leq & \left|m\left(\Im P\left(t\right)\ind\right)\right|+\left|m\left(\Im P\left(t\right)\psi\left(t\right)\right)\right|,
\end{eqnarray*}
where $\psi\left(t\right)=\ind-\eta\left(t\right)$. By definition
of the characteristic function operator, and the fact that $\hat{T}f$
is real if $f$ is real, 
\begin{eqnarray*}
\left|m\left(\Im P\left(t\right)\psi\left(t\right)\right)\right| & \leq & \left|m\left(\hat{T}\left(\cos\left(t\varphi\right)\Im\psi\left(t\right)\right)\right)\right|+\left|m\left(\hat{T}\left(\sin\left(t\varphi\right)\Re\psi\left(t\right)\right)\right)\right|.
\end{eqnarray*}
 Since by (\ref{eq: Extra condition}) $m\left(\psi\left(t\right)\right)=0$
, $m\circ\hat{T}=m$, $\left|1-\cos t\varphi\right|\leq t^{2}\varphi^{2}$,
$\left|\psi\left(t\right)\right|\apprle\left|t\right|$ and by the
positivity of the transfer operator,
\begin{eqnarray*}
\left|m\left(\hat{T}\left(\cos\left(t\varphi\right)\Im\psi\left(t\right)\right)\right)\right| & = & \left|m\left(\left(\cos\left(t\varphi\right)-1\right)\Im\psi\left(t\right)\right)\right|\\
 & \leq & m\left(t^{2}\varphi^{2}\left\Vert \psi\left(t\right)\right\Vert \right)\\
 & \apprle & \left|t\right|^{3}
\end{eqnarray*}
where we have used the finiteness of the second moment of $\varphi.$ 

Since $\psi\left(0\right)=0$, $\Re\psi'\left(0\right)=0$ (because
$\eta'\left(0\right)$ is purely imaginary) and $\psi\left(t\right)$
is twice continuously differentiable, 
\[
\left|m\left(\hat{T}\left(\sin\left(t\varphi\right)\Re\psi\left(t\right)\right)\right)\right|\apprle\left|t\right|^{3}.
\]
Therefore, 
\begin{eqnarray*}
\sum_{n=1}^{\infty}\int_{C_{\delta}}n\left|\lambda\left(t\right)\right|^{n-1}\left|m\left(\Im P\left(t\right)\psi\left(t\right)\right)\right|\left|\sin\left(ty\right)\right|dt & \apprle & \sum_{n=1}^{\infty}\int_{C_{\delta}}n\left(1-ct^{2}\right)^{n-1}\left|t\right|^{3}\left|\sin ty\right|dt\\
 & \leq & \int_{C_{\delta}}\frac{1}{ct^{4}}\left|t\right|^{3}\left|\sin ty\right|dt\\
 & \apprle & \left|y\right|
\end{eqnarray*}
Finally, since $m\left(\varphi\right)=0$ and $m\circ\hat{T}=m$ 
\begin{eqnarray*}
\left|m\left(\Im P\left(t\right)\ind\right)\right| & = & \left|m\left(\sin\left(t\varphi\right)\right)\right|\\
 & = & \left|m\left(\sin\left(t\varphi\right)-t\varphi\right)\right|
\end{eqnarray*}
We split the last integral into parts where $\left|t\varphi\right|\leq1$
and $\left|t\varphi\right|>1$ to obtain 
\begin{eqnarray*}
\left|m\left(\Im P\left(t\right)\ind\right)\right| & \leq & \left|m\left(\ind_{\left\{ \left|t\varphi\right|\leq1\right\} }\left(\sin\left(t\varphi\right)-t\varphi\right)\right)\right|+\left|m\left(\ind_{\left\{ \left|t\varphi\right|>1\right\} }\left(\sin\left(t\varphi\right)-t\varphi\right)\right)\right|\\
 & \leq & \left|m\left(\ind_{\left\{ \left|t\varphi\right|\leq1\right\} }\left|t\varphi\right|^{3}\right)\right|+\left|m\left(2\left|t\varphi\right|\ind_{\left\{ \left|t\varphi\right|>1\right\} }\right)\right|
\end{eqnarray*}
Thus, summing over $n$ and again using $\left|\lambda\left(t\right)\right|\leq\left(1-ct^{2}\right)$
we have
\begin{align}
\sum_{n=1}^{\infty}\int_{C_{\delta}}n\left|\lambda\left(t\right)\right|^{n-1}\left|m\left(\Im P\left(t\right)\ind\right)\right| & \left|\sin\left(ty\right)\right|dt\label{eq: eq 2}\\
\leq & \int_{C_{\delta}}\frac{1}{ct^{4}}m\left(\left|t\varphi\right|^{3}\ind_{\left\{ \left|t\varphi\leq1\right|\right\} }\right)\left|\sin\left(ty\right)\right|dt\nonumber \\
 & +2\int_{C_{\delta}}\frac{1}{ct^{4}}m\left(\left|t\varphi\right|\ind_{\left\{ \left|t\varphi\right|>1\right\} }\right)\left|\sin\left(ty\right)\right|dt\nonumber 
\end{align}
 Bounding $\left|\sin ty\right|$ by $\left|ty\right|$ and changing
the order of integration in the first term gives 
\begin{eqnarray}
\int_{C_{\delta}}\frac{1}{ct^{4}}m\left(\left|t\varphi\right|^{3}\ind_{\left\{ \left|t\varphi\leq1\right|\right\} }\right)\left|\sin ty\right|dt & \leq & m\left(\left|\varphi\right|^{3}\int_{-\left|\varphi\right|^{-1}}^{\left|\varphi\right|^{-1}}\left|y\right|dt\right)\label{eq: 1}\\
 & = & m\left(2\left|\varphi\right|^{2}\right)\left|y\right|\nonumber \\
 & \apprle & \left|y\right|\nonumber 
\end{eqnarray}
Changing the order of integration in the second term of (\ref{eq: eq 2})
and using the fact the the integrand is an even function of $t$,
gives 
\begin{eqnarray*}
2\int_{C_{\delta}}\frac{1}{ct^{4}}m\left(\left|t\varphi\right|\ind_{\left\{ \left|t\varphi\right|>1\right\} }\right)\left|\sin\left(ty\right)\right|dt & \leq & 4m\left(\left|\varphi\right|\int_{\left|\varphi\right|^{-1}}^{\delta}\frac{1}{t^{2}}\left|y\right|dt\right)\\
 & \lesssim & \left|y\right|.
\end{eqnarray*}
This completes the proof.\end{proof}
\begin{prop}
\label{prop: 4th moment ineq}Let $\delta>0$, $n\in\bbN$. Then there
exists a constant $C>0$ such that for all $x,y\in\bbR$ such that
\[
m\biggl(\Bigl(\sum_{k=1}^{n}f\left(S_{n}-x\right)-f\left(S_{n}-y\right)\Bigr)^{4}\biggr)\leq C\left(n\left|x-y\right|^{2}\right).
\]
\end{prop}
\begin{proof}
Opening brackets we obtain 
\begin{equation}
m\left(\left(\sum_{k=1}^{n}f\left(S_{k}-x\right)-f\left(S_{k}-y\right)\right)^{4}\right)=\sum_{\left(k_{1},...,k_{4}\right)\in\left\{ 1,...n\right\} ^{4}}m\left(\prod_{l=1}^{4}\left(f\left(S_{k_{l}}-x\right)-f\left(S_{k_{l}}-y\right)\right)\right)\label{eq: Moment ineq 1-1}
\end{equation}
It is clear that that the left hand term is not greater than $4!$
times the same sum over all tuples $\left(k_{1},...,k_{4}\right)\in\left\{ 1,...,n\right\} ^{4}$
where $k_{1}\leq k_{2}\leq k_{3}\leq k_{4}$. Thus we assume that
$k_{1}\leq k_{2}\leq k_{3}\leq k_{4}$. By inversion formula for the
Fourier transform 
\begin{eqnarray*}
m\biggl(\prod_{l=1}^{4}\bigl(f\bigl(S_{k_{l}}-x\bigr)-f\bigl(S_{k_{l}}-y\bigr)\bigr)\biggr) & = & m\biggl(\int\limits _{\bbR^{4}}\prod_{l=1}^{4}\hat{f}\left(t_{l}\right)e^{it_{l}S_{k_{l}}}\left(e^{it_{l}x}-e^{it_{l}y}\right)dt_{1}...dt_{4}\biggr)\\
 & = & \int\limits _{\bbR^{4}}\prod_{l=1}^{4}\hat{f}\left(t_{l}\right)\left(e^{it_{l}x}-e^{it_{l}y}\right)m\Bigl(\prod_{l=1}^{4}e^{it_{l}S_{k_{l}}}\Bigr)dt_{1}...d_{t_{4}}
\end{eqnarray*}
where the last equality follows follows by changing order of integration. 

Writing $k_{0}=0$, by definition of the characteristic function operator
we have 
\begin{eqnarray*}
m\Bigl(\prod_{l=1}^{4}e^{it_{l}S_{k_{l}}}\Bigr) & = & m\Bigl(\prod_{l=1}^{4}e^{i\sum_{j=l}^{4}t_{l}\bigl(S_{k_{l}}-S_{k_{l-1}}\bigr)}\Bigr)\\
 & = & m\biggl(\biggl(\prod_{l=1}^{4}P^{k_{l}-k_{l-1}}\Bigl(\sum_{j=l}^{4}t_{j}\Bigr)\biggr)\left(\ind\right)\biggr).
\end{eqnarray*}
Thus, 

\[
\begin{aligned}m\Bigl(\prod_{l=1}^{4}\Bigl(f\bigl(S_{k_{l}}-x\bigr)-f & \bigl(S_{k_{l}}-y\bigr)\Bigr)\Bigr)\\
 & =\int\limits _{\bbR^{4}}m\Bigl(\prod_{l=1}^{4}P^{k_{l}-k_{l-1}}\Bigl(\sum_{j=l}^{4}t_{j}\Bigr)\Bigr)\bigl(\ind\bigr)\Bigr)\Bigr)\cdot\prod_{l=1}^{4}\hat{f}\left(t_{l}\right)\left(e^{it_{l}x}-e^{it_{l}y}\right)dt_{1}...d_{4}.
\end{aligned}
\]
Performing a change of variables $z_{i}=\sum_{k=1}^{i}t_{i}$ $i=1,...,4$,
and writing $t_{0}=0$, $n_{l}=k_{l}-k_{l-1}$ we obtain 

\[
\begin{aligned}m\Bigl(\prod_{l=1}^{4}\Bigl(f & \bigl(S_{k_{l}}-x\bigr)-f\bigl(S_{k_{l}}-y\bigr)\Bigr)\Bigr)\\
 & =\int\limits _{\bbR^{4}}m\left(\prod_{l=1}^{4}P^{n_{l}}\left(t_{i}\right)\left(\ind\right)\right)\prod_{l=1}^{4}\hat{f}\left(t_{l}\right)\left(e^{i\left(t_{l}-t_{l-1}\right)x}-e^{i\left(t_{l}-t_{l-1}\right)y}\right)dt_{1}...dt_{4}
\end{aligned}
\]
Next, we need to simplify the expression $\prod_{l=1}^{4}\left(e^{i\left(t_{l}-t_{l-1}\right)x}-e^{i\left(t_{l}-t_{l-1}\right)y}\right)$.
$ $Let $\zeta\left(x\right)=y-x$ and $\zeta\left(y\right)=x-y$.
We claim that 
\begin{equation}
\prod_{l=1}^{4}\left(e^{i\left(t_{l}-t_{l-1}\right)x}-e^{i\left(t_{l}-t_{l-1}\right)y}\right)=\sum_{\left(z_{2},z_{4}\right)}e^{it_{2}z_{2}}\left(1-e^{it_{1}\zeta\left(z_{2}\right)}\right)e^{it_{4}z_{4}}\left(1-e^{it_{3}\zeta\left(z_{4}\right)}\right)\label{eq:momnet ineq 3-1}
\end{equation}
where the sum is over all $\left(z_{2},z_{4}\right)\in\left\{ x,y\right\} ^{2}$.
To see this note that 
\[
\left(e^{it_{l}x}-e^{it_{l}y}\right)\left(e^{i\left(t_{l+1}-t_{l}\right)x}-e^{i\left(t_{l+1}-t_{l}\right)y}\right)=e^{it_{l+1}x}\left(1-e^{it_{l}\left(x-y\right)}\right)+e^{it_{l+1}y}\left(1-e^{it_{l}\left(y-x\right)}\right)
\]
and the claim follows by implementing this on the first two terms
and the last two terms in the product on the left hand side of (\ref{eq:momnet ineq 3-1})
separately.

To shorten the writing we write $\psi\left(t,z\right)=1-e^{it\zeta\left(z\right)}$.
Thus, 

\[
\begin{aligned}m\Bigl(\bigl(\sum_{k=1}^{n}f\bigl(S_{k}-x\bigr) & -f\bigl(S_{k}-y\bigr)\bigr)^{4}\Bigr)\\
 & \leq4!\sum\limits _{n_{1,...,}n_{4}}\Re\Bigl[\int_{\bbR^{4}}\prod_{l=1}^{4}\hat{f}\bigl(t_{l}\bigr)\cdot m\Bigl(\prod_{l=1}^{4}P^{n_{l}}\left(t_{l}\right)\left(\ind\right)\Bigr)\\
 & \times\sum\limits _{\left(z_{2},z_{4}\right)}e^{it_{2}z_{2}}\psi\left(t_{1},z_{2}\right)e^{it_{4}z_{4}}\psi\left(t_{3},z_{4}\right)dt_{1}...dt_{4}\Bigr]
\end{aligned}
\]
where the sum is over all tuples $\left(n_{1},...,n_{4}\right)\in\left\{ 0,...,n\right\} ^{4}$.
The following inequality completes the proof. 

\begin{align}
\sum\limits _{n_{1,...,}n_{4}}\Re\int\limits _{\bbR^{4}}\prod_{l=1}^{4}\hat{f}\bigl(t_{l}\bigr)\cdot & m\Bigl(\prod_{l=1}^{4}P^{n_{l}}\left(t_{l}\right)\left(\ind\right)\Bigr)\sum\limits _{\left(z_{2},z_{4}\right)}e^{it_{2}z_{2}}\psi\left(t_{1},z_{2}\right)e^{it_{4}z_{4}}\psi\left(t_{3},z_{4}\right)dt_{1}...dt_{4}\label{eq: moment ineq 5-1}\\
 & \apprle n\left|x-y\right|^{2}\nonumber 
\end{align}
To prove this we use the expansion of the characteristic function
operator in proposition \ref{prop:Perturbation theorem} and lemma
\ref{lem: Perturbation theorem 1} and propositions \ref{prop:LLT},\ref{prop: Potential kernel }.
Let $\delta$ be as in lemma (\ref{eq: Expansion of eigenvalue})
and write $C_{\delta}^{i}=\left[-\delta,\delta\right]^{i}$ and $\bar{C}_{\delta}=\bbR^{i}\setminus C_{\delta}^{i}$.
Also denote $\zeta\left(t\right)f=m\left(f\right)\ind-\pi\left(t\right)f$
and recall from \ref{lem: Perturbation theorem 1} that $\left\Vert \zeta\left(t\right)\right\Vert \apprle\left|t\right|$
for $\left|t\right|<\delta$. For fixed $\left(z_{2},z_{4}\right)\in\left\{ x,y\right\} ^{2}$,

\begin{equation}
\begin{aligned}\int\limits _{C_{\delta}^{4}}\prod_{l=1}^{4} & \hat{f}\bigl(t_{l}\bigr)\cdot m\Bigl(\prod_{l=1}^{4}P^{n_{l}}\left(t_{l}\right)\left(\ind\right)\Bigr)e^{it_{2}z_{2}}\psi\left(t_{1},z_{2}\right)e^{it_{4}z_{4}}\psi\left(t_{3},z_{4}\right)dt_{1}...dt_{4}\\
 & =\int\limits _{C_{\delta}^{4}}\prod_{l=1}^{4}\hat{f}\bigl(t_{l}\bigr)\cdot\lambda^{n_{4}}\left(t_{4}\right)m\left(\left(\prod_{l=1}^{3}P^{n_{l}}\left(t_{l}\right)\right)\left(\ind\right)\right)e^{it_{2}z_{2}}\psi\left(t_{1},z_{2}\right)e^{it_{4}z_{4}}\psi\left(t_{3},z_{4}\right)dt_{1}...dt_{4}\\
 & +\int\limits _{C_{\delta}^{4}}\prod_{l=1}^{4}\hat{f}\bigl(t_{l}\bigr)\lambda^{n_{4}}\left(t_{4}\right)m\left(\zeta\left(t_{4}\right)\left(\prod_{l=1}^{3}P^{n_{l}}\left(t_{l}\right)\right)\left(\ind\right)\right)e^{it_{2}z_{2}}\psi\left(t_{1},z_{2}\right)e^{it_{4}z_{4}}\psi\left(t_{3},z_{4}\right)dt_{1}...dt_{4}\\
 & +\int\limits _{C_{\delta}^{4}}\prod_{l=1}^{4}\hat{f}\bigl(t_{l}\bigr)N^{n_{4}}\left(t_{4}\right)\prod_{l=1}^{3}P^{n_{l}}\left(t_{l}\right)e^{it_{2}z_{2}}\psi\left(t_{1},z_{2}\right)e^{it_{4}z_{4}}\psi\left(t_{3},z_{4}\right)dt_{1}...dt_{4}
\end{aligned}
\label{eq: moment estimate 6-1}
\end{equation}
The proof is conducted similarly for all terms. We continue to expand
the products $\prod_{l=1}^{3}P^{n_{l}}\left(t_{l}\right)$ using $P^{n}\left(t\right)\left(\cdot\right)=\lambda^{n}\left(t\right)m\left(\cdot\right)\ind+\zeta\left(t\right)\left(\cdot\right)+N^{n}\left(t\right)\left(\cdot\right)$.
After this we split the integrals into a product of integrals. For
example we can split the first term on the right of (\ref{eq: moment estimate 6-1})
into 

\begin{align*}
\int\limits _{C_{\delta}^{4}}\prod_{l=1}^{4}\hat{f}\bigl(t_{l}\bigr) & \cdot\lambda^{n_{4}}\left(t_{4}\right)m\left(\left(\prod_{l=1}^{3}P^{n_{l}}\left(t_{l}\right)\right)\left(\ind\right)\right)e^{it_{2}z_{2}}\psi\left(t_{1},z_{2}\right)e^{it_{4}z_{4}}\psi\left(t_{3},z_{4}\right)dt_{1}...dt_{4}\\
= & \int\limits _{C_{\delta}^{4}}\prod_{l=1}^{4}\hat{f}\bigl(t_{l}\bigr)\lambda^{n_{4}}\left(t_{4}\right)e^{it_{4}z_{4}}dt_{4}\\
 & \times m\left(\left(\prod_{l=1}^{3}P^{n_{l}}\left(t_{l}\right)\right)\left(\ind\right)\right)e^{it_{2}z_{2}}\psi\left(t_{1},z_{2}\right)e^{it_{4}z_{4}}\psi\left(t_{3},z_{4}\right)dt_{1}..dt_{3}.
\end{align*}
For the other terms we may not split the integrals right away because
we have operators of the form $\zeta\left(t\right)\circ\zeta\left(s\right)$
or $N^{n}\left(t\right)\circ N^{k}\left(s\right)$. To handle these
kind of terms we continue to expand, and eventually will be able to
split the integrals by taking norms (splitting first the terms similar
to the above term prior to taking absolute values). Since by the proof
of proposition \ref{prop: Potential kernel } we have 
\[
\sum_{k=0}^{\infty}\int_{C_{\delta}}\left|\hat{f}\left(t\right)\right|\left|t\right|\left|\lambda^{k}\left(t\right)\right|\left(1-e^{it\left(x-y\right)}\right)dt\apprle\int_{C_{\delta}}\frac{\left|t\right|^{2}}{ct^{2}}\left|x-y\right|dt\apprle\left|x-y\right|
\]
\[
\sum_{n=0}^{\infty}\int_{C_{\delta}}\left|\hat{f}\left(t\right)\right|\left\Vert N^{k}\left(t\right)\right\Vert \left(1-e^{it\left(x-y\right)}\right)dt\apprle\left|x-y\right|
\]
and 
\[
\sum_{n=0}^{\infty}\left|\int_{C_{\delta}}\hat{f}\left(t\right)\lambda^{k}\left(t\right)\left(1-e^{it\left(x-y\right)}\right)dt\right|\apprle\left|x-y\right|
\]
the sum of every term involving a function $\psi$ is bounded by a
constant multiplied by $\left|x-y\right|$. 

On the other hand, by proposition \ref{prop:LLT} each term of the
form 
\[
\int_{C_{\delta}}\left|\hat{f}\left(t\right)\lambda^{k}\left(t\right)\right|\left|e^{itx}\right|dt
\]
\[
\int_{C_{\delta}}\left|\hat{f}\left(t\right)\lambda^{k}\left(t\right)\right|\left|t\right|\left|e^{itx}\right|dt
\]
\[
\int_{C_{\delta}}\left|\hat{f}\left(t\right)\right|\left\Vert N^{n}\left(t\right)\right\Vert dt
\]
is bounded by a constant multiplied by $\frac{1}{\sqrt{n}}$. It follows
that the sum from $0$ to $n$ of such terms is bounded by $\sqrt{n}$.
Since the integral in \ref{eq: moment ineq 5-1} has precisely two
terms involving functions $\psi$ it follows that
\[
\sum_{n_{1},...,n_{4}}\int\limits _{C_{\delta}^{4}}\prod_{l=1}^{4}\hat{f}\bigl(t_{l}\bigr)\cdot m\Bigl(\prod_{l=1}^{4}P^{n_{l}}\left(t_{l}\right)\left(\ind\right)\Bigr)e^{it_{2}z_{2}}\psi\left(t_{1},z_{2}\right)e^{it_{4}z_{4}}\psi\left(t_{3},z_{4}\right)dt_{1}...dt_{4}\apprle\left|x-y\right|^{2}.
\]
 Estimating 
\[
\sum_{n_{1},...,n_{4}}\int_{\bar{C}_{\delta}^{4}}\prod_{l=1}^{4}\hat{f}\bigl(t_{l}\bigr)\cdot m\Bigl(\prod_{l=1}^{4}P^{n_{l}}\left(t_{l}\right)\left(\ind\right)\Bigr)e^{it_{2}z_{2}}\psi\left(t_{1},z_{2}\right)e^{it_{4}z_{4}}\psi\left(t_{3},z_{4}\right)dt_{1}...dt_{4}
\]
is easier since at least one of the four integrals at hand is over
$\bbR\setminus\left(-\delta,\delta\right)$ and therefore, exponentially
tends to $0$ by lemma \ref{lem: Exponential decay on compacts}.
Expanding this integral similarly to the integral over $\bar{C}_{\delta}$
we obtain a similar estimate
\[
\sum\int_{\bar{C}_{\delta}^{4}}\prod_{l=1}^{4}\hat{f}\bigl(t_{l}\bigr)\cdot m\Bigl(\prod_{l=1}^{4}P^{n_{l}}\left(t_{l}\right)\left(\ind\right)\Bigr)e^{it_{2}z_{2}}\psi\left(t_{1},z_{2}\right)e^{it_{4}z_{4}}\psi\left(t_{3},z_{4}\right)dt_{1}...dt_{4}\apprle\left|x-y\right|^{2}
\]
whence the proposition follows. \end{proof}
\begin{cor}
\label{cor: fourth moment ineq for l_n}There exists a constant $C$
such that for all $n\in\bbN$, $x,y\in\bbR$, 
\[
m\left(\left(\left(l_{n}\left(x\right)-l_{n}\left(y\right)\right)\right)^{4}\right)\leq C\left|x-y\right|^{2}.
\]
\end{cor}
\begin{proof}
By proposition \ref{cor: fourth moment ineq for l_n} 
\begin{eqnarray*}
m\left(\left(\left(l_{n}\left(x\right)-l_{n}\left(y\right)\right)\right)^{4}\right) & = & \frac{1}{n^{2}}m\left(\left(\sum_{k=0}^{n}f\left(S_{k}-\sqrt{n}x\right)-f\left(S_{k}-\sqrt{n}y\right)\right)^{4}\right)\\
 & \leq & \frac{C}{n^{2}}\left(n^{2}\left|x-y\right|^{2}\right)\\
 & = & C\left|x-y\right|^{2}.
\end{eqnarray*}
\end{proof}
\begin{prop}
\label{prop:Variance ineq}There exists a constant $C$ such that
for all $x\in\bbR$, $n\in\bbN$, 
\[
m\left(\left(\sum_{k=1}^{n}f\left(S_{k}-x\right)\right)^{2}\right)\leq Cn.
\]
\end{prop}
\begin{proof}
Using similar methods to proposition \ref{prop: 4th moment ineq}
we have
\begin{eqnarray*}
m\left(\left(\sum_{k=1}^{n}f\left(S_{n}-x\right)\right)\right) & = & \sum_{k,l=1}^{n}m\left(f\left(S_{l}-x\right)f\left(S_{k}-x\right)\right)\\
 & \leq & 2!\sum_{l=1}^{n}\sum_{k=l}^{n}m\left(f\left(S_{l}-x\right)f\left(S_{k}-x\right)\right)\\
 & \apprle & \sum_{l=1}^{n}\sum_{k=l}^{n}m\left(\int\limits _{\bbR^{2}}\hat{f}\left(t_{1}\right)\hat{f}\left(t_{2}\right)e^{it_{1}\left(S_{l}-x\right)}e^{it_{2}\left(S_{k}-x\right)}dt_{1}dt_{2}\right)\\
 & = & \sum_{l=1}^{n}\sum_{k=l}^{n}\int\limits _{\bbR^{2}}\hat{f}\left(t_{1}\right)\hat{f}\left(t_{2}\right)P^{k-l}\left(t_{1}\right)P^{l}\left(t_{2}+t_{1}\right)\left(\ind\right)e^{-i\left(t_{1}+t_{2}\right)x}dt_{1}dt_{2}\\
 & \leq & \sum_{l=1}^{n}\sum_{k=l}^{n}\int_{\bbR^{2}}\left|\hat{f}\left(t_{1}\right)\hat{f}\left(t_{2}\right)\right|\Bigl\Vert P^{k-l}\left(t_{1}\right)\Bigr\Vert\Bigl\Vert P^{l}\left(t_{2}\right)\Bigr\Vert dt_{1}dt_{2}
\end{eqnarray*}
where the last inequality follows by a change of variables and taking
absolute values. Writing $C_{\delta}=\left(-\delta,\delta\right)^{2}$
, $\bar{C}_{\delta}=\bbR^{2}\setminus\left(-\delta,\delta\right)^{2}$
and taking $\delta$ as in lemma \ref{lem: Perturbation theorem 1},
we obtain by proposition \ref{prop:LLT}

\begin{eqnarray*}
\sum_{l=1}^{n}\sum_{k=l}^{n}\int_{C_{\delta}}\left|\hat{f}\left(t_{1}\right)\hat{f}\left(t_{2}\right)\right|\Bigl\Vert P^{l-k}\left(t_{1}\right)\Bigr\Vert\Bigl\Vert P^{k}\left(t_{2}\right)\Bigr\Vert dt_{1}dt_{2} & \lesssim & \sum_{l=1}^{n}\sum_{k=l}^{n}\int_{-\delta}^{\delta}\left|\lambda^{l-k}\left(t_{1}\right)\right|dt_{1}\int_{-\delta}^{\delta}\left|\lambda^{k}\left(t_{2}\right)\right|dt_{2}\\
 & \leq & n.
\end{eqnarray*}
By the use of lemma \ref{lem: Exponential decay on compacts} and
proposition \ref{prop:LLT}
\[
\sum_{l=1}^{n}\sum_{k=l}^{n}\int_{\bar{C}_{\delta}}\left|\hat{f}\left(t_{1}\right)\hat{f}\left(t_{2}\right)\right|\Bigl\Vert P^{l-k}\left(t_{1}\right)\Bigr\Vert\Bigl\Vert P^{k}\left(t_{2}\right)\Bigr\Vert dt_{1}dt_{2}\lesssim n,
\]
which completes the proof. \end{proof}
\begin{cor}
\label{cor:Variance ineq}There exists a constant $C$ such that for
all $n\in\bbN$, $x\in\bbR$, $m\left(\left(l_{n}\left(x\right)\right)^{2}\right)\leq C$.\end{cor}
\begin{proof}
By proposition \ref{prop:Variance ineq} 
\[
m\left(\left(l_{n}\left(x\right)\right)^{2}\right)=\frac{1}{n}m\left(\left(\sum_{k=1}^{n}f\left(S_{n}-\sqrt{n}x\right)\right)^{2}\right)\leq C.
\]
To prove tightness of the process $l_{n}\left(x\right)$, additionally
to the above estimates we need the following estimate for maxima of
continuous processes. \end{proof}
\begin{prop}
\label{prop: Maxima ineq}Let $\gamma\left(t\right)$ be an almost
surely continuous process on an interval $I$ of length $\delta$.
Assume that for every $\epsilon>0$, $t,s\in I$ we have $P\left(\left|\gamma\left(t\right)-\gamma\left(s\right)\right|\geq\epsilon\right)<C\frac{\left|t-s\right|^{\alpha}}{\epsilon^{\beta}}$,
where $\alpha>1$, $\beta\geq0$, $C>0$. Then $P\left(\sup_{t,s\in I}\left|\gamma\left(t\right)-\gamma\left(s\right)\right|\geq\epsilon\right)\leq\tilde{C}\frac{\delta^{\alpha}}{\epsilon^{\beta}}$. \end{prop}
\begin{proof}
Assume without loss of generality that $I=\left[0,\delta\right]$
and let $D_{k}:=\left\{ 0,\frac{1}{2^{k}}\delta,\frac{2}{2^{k}}\delta,...,\frac{2^{k-1}}{2^{k}}\delta,\delta\right\} $.
Let $B_{k}=\max_{t_{1},t_{2}\in D_{k}}\left|\gamma\left(t_{1}\right)-\gamma\left(t_{2}\right)\right|$
and let $A_{k}=\max_{t_{1},t_{2}\in B_{k},\left|t_{1}-t_{2}\right|=\frac{1}{2^{k}}\delta}$.
For $t\in D_{k}$, define a point $t'\in D_{k-1}$ by 
\[
t'=\begin{cases}
t & t\in D_{k-1}\\
t-2^{-k} & t\notin D_{k-1}
\end{cases}.
\]
Then for every $t\in D_{k}$, $\left|\gamma\left(t\right)-\gamma\left(t'\right)\right|\leq A_{k}$
and therefore, for $t_{1},t_{2}\in D_{k}$ we have
\begin{eqnarray*}
\left|\gamma\left(t_{1}\right)-\gamma\left(t_{2}\right)\right| & \leq & \left|\gamma\left(t_{1}\right)-\gamma\left(t_{1}'\right)\right|+\left|\gamma\left(t_{1}'\right)-\gamma\left(t_{2}'\right)\right|+\left|\gamma\left(t_{2}'\right)-\gamma\left(t_{2}\right)\right|\\
 & \leq & \left|\gamma\left(t_{1}'\right)-\gamma\left(t_{2}'\right)\right|+2A_{k}.
\end{eqnarray*}
Since $t_{1}'$, $t_{2}'$ are in $D_{k-1}$, it follows that $B_{k}\leq B_{k-1}+2A_{k}$.
Since $A_{0}=B_{0}$, we conclude by induction that $B_{k}\leq2\sum_{i=1}^{k}A_{i}$.
By continuity of the paths of $\gamma$, we get that $\lim_{k\rightarrow\infty}B_{k}=\sup_{t,s\in I}\left|\gamma\left(t\right)-\gamma\left(s\right)\right|$
and therefore, 
\[
\sup_{t.s\in I}\left|\gamma\left(t\right)-\gamma\left(s\right)\right|\leq2\sum_{k=1}^{\infty}A_{k}.
\]
Suppose that $\theta\in\left(0,1\right)$ and let $r$ be such that
$r\cdot\sum_{k=1}^{\infty}\theta^{k}=\frac{1}{2}$. Then 
\begin{eqnarray*}
P\left(\sup_{t.s\in I}\left|\gamma\left(t\right)-\gamma\left(s\right)\right|\geq\epsilon\right) & \leq & P\left(2\sum_{k=1}^{\infty}A_{k}\geq\epsilon\right)\\
 & \leq & \sum_{k=1}^{\infty}P\left(A_{k}\geq r\epsilon\theta^{k}\right)\\
 & \leq & \sum_{k=1}^{\infty}C2^{k}\left(\frac{\delta}{2^{k}}\right)^{\alpha}\frac{1}{\left(r\epsilon\theta^{k}\right)^{\beta}}\\
 & = & C\frac{\delta^{\alpha}}{\left(r\epsilon\right)^{\beta}}\sum_{k=1}^{\infty}\frac{1}{\left(2^{\alpha-1}\theta^{\beta}\right)^{k}}.
\end{eqnarray*}
Since $\alpha-1>0$ and $\beta\geq0$, there exists $\theta$ for
which the sum converges and the claim follows. 
\end{proof}

\section{\label{sec:Tightness}Tightness of $l_{n}$ in $D$.}

A sequence $\left\{ X_{n}\right\} $ of random variables taking values
in a complete and separable metric space $\left(X,d\right)$ is tight
if for every $\epsilon>0$ there exists a compact $K\subset X$ such
that for every $n\in\bbN$,
\[
P_{n}(K)>1-\epsilon,
\]
where $P_{n}$ denotes the distribution of $X_{n}$ . By Prokhorov's
Theorem (see \cite{Bil}) relative compactness of $t_{n}(x)$ in $C$
is equivalent to tightness. Therefore, we are interested in characterizing
tightness in $C$.

For $h>0$, denote by $C_{\left[-h,h\right]}$ the space of continuous
functions on $C_{\left[-h,h\right]}$. For $x\left(t\right)$ in $C_{\left[-h,h\right]},$
set

\[
\omega_{x}(\delta):=\sup_{\left|s-t\right|<\delta}\left\{ \left|x\left(s\right)-x\left(t\right)\right|:s,t\in C_{\left[-h,h\right]},\,\left|s-t\right|<\delta\right\} .
\]
$\omega_{x}\left(\delta\right)$ is called the modulus of continuity
of $x$. Due to the Arzela - Ascoli theorem, the modulus of continuity
plays a central role in characterizing precompactness in the space
$C_{\left[-h,h\right]}$, with the Borel $\sigma$-algebra generated
by the topology of uniform convergence. 

The next theorem is a characterization of tightness in the space $C$. 
\begin{thm}
\label{thm:-characterization of tightness}\cite{Bil} The sequence
$l_{n}$ is tight in $C$ if and only if its restriction to $\left[-h,h\right]$
is tight in $C_{\left[-h,h\right]}$ for every $h\in\bbR_{+}$. The
sequence $l_{n}$ is tight in $C_{\left[-h,h\right]}$ if and only
if the following two conditions hold:

(i) $\forall x\in\left[-h,h\right],\;\lim\limits _{a\rightarrow\infty}\limsup\limits _{n\rightarrow\infty}m\left(\left|l{}_{n}\left(x\right)\right|\geq a\right)=0.$

(ii) $\forall\epsilon>0,\,\lim\limits _{\delta\rightarrow0}\limsup\limits _{n\rightarrow\infty}m\left(\omega_{l_{n}}\left(\delta\right)\geq\epsilon\right)=0.$ \end{thm}
\begin{prop}
\label{pro:The-sequence is tight}The sequence $\left\{ l_{n}\right\} _{n=1}^{\infty}$
is tight. \end{prop}
\begin{proof}
Condition $(i)$ of theorem \ref{thm:-characterization of tightness}
easily follows from corollary \ref{cor:Variance ineq} and Chebychev's
inequality, since for all $x\in\bbR$, $n\in\bbN$, 
\[
m\left(\left|l_{n}\left(x\right)\right|\geq a\right)\leq\frac{m\left(\left(l_{n}\left(x\right)\right)^{2}\right)}{a^{2}}\leq\frac{C}{a^{2}}\underset{a\rightarrow\infty}{\longrightarrow}0.
\]
We prove that condition $(ii)$ of theorem \ref{thm:-characterization of tightness}
holds. In order to do that, we have to show that for fixed $h>0$,
\begin{equation}
\forall\epsilon>0.\ \lim\limits _{\delta\rightarrow0}\limsup\limits _{n\rightarrow\infty}m\left[\sup_{x,y\in[-h.h];\left|x-y\right|<\delta}\left|l_{n}(x)-l_{n}(y)\right|\geq\epsilon\right]=0.\label{eq: sufficient condition for tightness}
\end{equation}
Fix $\epsilon>0$. Then by \ref{cor: fourth moment ineq for l_n}
and Chebychev's inequality there exists a constant $C$ such that
for all $x,y\in\left[-h,h\right]$, 
\[
m\left(\left|l_{n}\left(x\right)-l_{n}\left(y\right)\right|\geq\epsilon\right)\leq\frac{m\left(\left(l_{n}\left(x\right)-l_{n}\left(y\right)\right)^{4}\right)}{\epsilon^{2}}\leq\frac{C\left|x-y\right|^{2}}{\epsilon^{2}}.
\]
Thus, for $x,y\in\left[-h,h\right]$ we have,

\[
m\left(\left|l_{n}\left(x\right)-l_{n}\left(y\right)\right|\geq\epsilon\right)\leq C\frac{\left|x-y\right|^{2}}{\epsilon^{2}}
\]
Let $\delta>0$ and $n>\delta^{-2}$. Then by proposition \ref{prop: Maxima ineq}
\begin{eqnarray*}
m\left(\sup_{x,y\in[-h.h];\left|x-y\right|<\delta}\left|l_{n}(x)-l_{n}(y)\right|\geq4\epsilon\right) & \leq & \sum\limits _{|k\delta|\leq h}m\left(\sup\limits _{k\delta\sqrt{n}\leq x,y\leq\left(k+1\right)\delta\sqrt{n}}\left|l_{n}\left(x\right)-l_{n}\left(y\right)\right|\geq\epsilon\right)\\
 & \leq & \sum_{\left|k\delta\right|\leq h}\frac{\tilde{C}}{\epsilon^{2}}\delta^{2}\leq2h\tilde{C}\delta\underset{\delta\rightarrow0}{\longrightarrow}0
\end{eqnarray*}
whence \ref{eq: sufficient condition for tightness} follows.
\end{proof}

\section{\label{Sec:Identifying the Only Possible Limit Point}Proof of the
main theorem}

In this section we identify $\left(\omega,\int_{\bbR}f\left(x\right)dx\cdot l\right)$
as the unique distributional limit of $\left(\omega_{n},l_{n}\right)$
and complete the proof of theorem \ref{sec:Main theorem}.

\subsection*{Proof of theorem \ref{thm:Main thm}:}

By assumption (A7) and proposition \ref{pro:The-sequence is tight},
$\left(\omega_{n},l_{n}\right)$ is tight in $D\left[0,1\right]\times C$.
Let $\left(p,q\right)$ be a distributional limit of some subsequence
$\left(\omega_{n_{k}},l_{n_{k}}\right)$. We must show that $\left(p,q\right)\overset{d}{=}\left(\omega,\int_{\bbR}f\left(x\right)dx\cdot l\right)$.
In what follows, we assume without loss of generality that the convergent
subsequence is $\left(\omega_{n},l_{n}\right)$ itself. By Skorokhod's
representation theorem there exists a probability space $\left(\varOmega,\mathcal{B},P\right)$
with random functions $\omega'_{n}$, $l'_{n}$, $p'$, $q'$ defined
on it, such that $\left(\omega'_{n},l'_{n}\right)\overset{d}{=}\left(\omega_{n},l_{n}\right)$,
$\left(p',q'\right)\overset{d}{=}\left(p,q\right)$ and $\left(\omega'_{n},l'_{n}\right)$$ $
almost surely converge to $\left(p,q\right)$$ $, i.e $d_{J}\left(\omega'_{n_{k}},p\right)\longrightarrow0$,
$d\left(l_{n_{k}},l\right)\longrightarrow0$ almost surely, where
$d_{J}$ is the metric of $D$ and $d$ is the metric of $C$. By
assumption that $\omega_{n}\limdist\omega$, we have $p\overset{d}{=}\omega$.
Let $G_{k}=\left\{ a_{1},b_{1},...,a_{k},b_{k}:\, a_{i}<b_{i},\: a_{i},b_{i}\in\bbQ,\: i=1,...,k\right\} $
and $G=\bigcup_{k=1}^{\infty}G_{k}$. For 
\[
g=\left\{ a_{1},b_{1},...,a_{k},b_{k}:\, a_{i}<b_{i},\: i=1,...,k\right\} \in G_{k}
\]
 define the transformation $\pi_{g}:C\rightarrow\bbR^{k}$ by 
\[
\pi_{g}\left(h\right)=\left(\int_{a_{1}}^{b_{1}}h\left(x\right)\, dx,...\int_{a_{n}}^{b_{n}}h\left(x\right)\, dx\right)
\]
and $\Pi_{g}:D\rightarrow\bbR^{k}$ by 
\[
\Pi_{g}\left(h\right)=\left(\int_{0}^{1}\ind_{\left[a_{1},b_{1}\right]}\left(h\right)dt,...,\int_{0}^{1}\ind_{\left[a_{k},b_{k}\right]}\left(q\left(t\right)\right)dt\right).
\]
Note that $\pi_{g}$ is continuous and therefore $\lim_{n\rightarrow\infty}\pi_{g}\left(l'_{n}\right)=\pi_{g}\left(q\right)$.
The following fact is proved in \cite{KS}: for all $k\in\bbN$, $g\in G_{k}$,
the function $\pi_{g}\left(\cdot\right)$ is continuous in the Skorokhod
topology at almost every sample point of the Brownian motion, i.e.
if $h_{n}\in D$ converges to $h\in D$ where $h$ is a generic sample
point of a Brownian motion, then $\Pi_{g}\left(h_{n}\right)\longrightarrow\Pi_{g}\left(h\right)$
(here, convergence is of vectors in $\bbR^{k}$). Since $G$ is a
countable set, this implies that almost surely, for all $g\in G$,
\begin{equation}
\Pi_{g}\left(\omega'_{n_{k}}\right)\longrightarrow\Pi_{g}\left(p\right)\label{eq:KS}
\end{equation}
To complete the proof, it is enough to show that for $g\in G_{k}$,
slmost surely 
\begin{equation}
\pi_{g}\left(q'\right)=\int_{\bbR}f\left(x\right)dx\cdot\pi_{g}\left(h\right)\label{eq: Ident. of limit}
\end{equation}
 where $h$ is the local time of the Brownian motion $p$. Since by
definition of local time the equality 
\[
\pi_{g}\left(h\right)=\Pi_{g}\left(p\right)
\]
 is almost surely satisfied for all $g\in G$ and since the transformations
$\pi_{g}\left(h\right)$, $g\in G$ uniquely determine the function
$h$, it would follow that almost surely $q$ coincides with $h$,
i.e that $q$ is the local time of the Brownian motion $p$ as claimed.
Setting $S_{k}^{n}:=\sqrt{n}\sum_{i=0}^{k}\omega_{n}\left(\frac{i}{n}\right)$
we have $l'_{n}\left(x\right)=\frac{1}{\sqrt{n}}\sum_{k=1}^{n}f\left(S_{k}^{n}\right)$.
Fix a positive $\epsilon\in\bbQ$, and $a,b\in\bbQ$, auch that $a<b$.
By straightforward calculations using a change of variable $y=S_{k}-\sqrt{n}x$,
we have 
\begin{eqnarray*}
\int\limits _{a}^{b}f\left(S_{k}^{n}-\sqrt{n}x\right)dx & = & n^{-\frac{1}{2}}\int\limits _{S_{k}-\sqrt{n}b}^{S_{k}-\sqrt{n}a}f\left(x\right)dx\\
 & \leq n^{-\frac{1}{2}} & \ind_{\left[a-\epsilon,b+\epsilon\right]}\left(n^{-\frac{1}{2}}S_{k}^{n}\right)\cdot\int\limits _{-\infty}^{\infty}f\left(x\right)dx\\
 &  & n^{-\frac{1}{2}}+\ind_{\bbR\setminus\left[a-\epsilon,b+\epsilon\right]}\left(n^{-\frac{1}{2}}S_{k}^{n}\right)\cdot\int\limits _{\sqrt{n}\epsilon}^{\infty}f\left(x\right)dx\\
 & \leq & n^{-\frac{1}{2}}\ind_{\left[a-\epsilon,b+\epsilon\right]}\left(n^{-\frac{1}{2}}S_{k}^{n}\right)\cdot\int\limits _{-\infty}^{\infty}f\left(x\right)dx+n^{-\frac{1}{2}}\int\limits _{\sqrt{n}\epsilon}^{\infty}f\left(x\right)dx.
\end{eqnarray*}
This implies 
\[
\int\limits _{a}^{b}l'_{n}\left(x\right)dx\leq\frac{1}{n}\#\left\{ k\in\left\{ 1,...,n\right\} :n^{-\frac{1}{2}}S_{k}^{n}\in\left[a-\epsilon,b+\epsilon\right]\right\} +\int\limits _{\sqrt{n}\epsilon}^{\infty}f\left(x\right)dx.
\]
On the other hand 
\[
\int_{0}^{1}\ind_{\left[a-\epsilon,b+\epsilon\right]}\left(\omega'_{n}\left(t\right)\right)dt=\frac{1}{n}\#\left\{ k\in\left\{ 1,...,n\right\} :S_{k}^{n}\in\left[a-\epsilon,b+\epsilon\right]\right\} 
\]
and by (\ref{eq:KS}) we have 
\[
\lim_{n\rightarrow\infty}\int_{0}^{1}\ind_{\left[a-\epsilon,b+\epsilon\right]}\left(\omega_{n}\left(t\right)\right)dt=\int_{0}^{1}\ind_{\left[a-\epsilon,b+\epsilon\right]}\left(p\left(t\right)\right)dt=\int_{a-\epsilon}^{b+\epsilon}h\left(x\right)dx\quad a.s
\]
 Since $\lim_{n\rightarrow\infty}\int_{\sqrt{n}\epsilon}^{\infty}f\left(x\right)dx=0$,
it follows that 
\[
\lim_{n\rightarrow\infty}\int_{a}^{b}l'_{n}\left(x\right)dx\leq\int\limits _{\bbR}f\left(x\right)dx\int_{a-\epsilon}^{b+\epsilon}h\left(x\right)dx
\]
and since this holds for every rational $\epsilon>0$, we have 
\[
\lim_{n\rightarrow\infty}\int\limits _{a}^{b}l'_{n}\left(x\right)dx=\int\limits _{a}^{b}q'\left(x\right)dx\leq\int\limits _{a}^{b}h\left(x\right)dx.
\]
To obtain a lower bound, imitating the calculations for the upper
bound, we have 
\[
\int\limits _{a}^{b}l'_{n}\left(x\right)dx\geq\frac{1}{n}\#\left\{ k\in\left\{ 1,...,n\right\} :n^{-\frac{1}{2}}S_{k}\in\left[a+\epsilon,b-\epsilon\right]\right\} \int\limits _{-\epsilon\sqrt{n}}^{\epsilon\sqrt{n}}f\left(x\right)dx
\]
and therefore
\[
\lim_{n\rightarrow\infty}\int\limits _{a}^{b}l_{n}\left(x\right)dx=\int\limits _{a}^{b}q\left(x\right)dx\geq\int_{\bbR}f\left(x\right)dx\cdot\int\limits _{a+\epsilon}^{b-\epsilon}l\left(x\right)dx.
\]
It follows that 
\[
\lim_{n\rightarrow\infty}\int_{a}^{b}l'_{n}\left(x\right)dx\geq\int\limits _{\bbR}f\left(x\right)dx\int\limits _{a}^{b}l\left(x\right)dx
\]
and therefore, 
\[
\lim_{n\rightarrow\infty}\int\limits _{a}^{b}l'_{n}\left(x\right)dx=\int\limits _{a}^{b}q'\left(x\right)=\int\limits _{\bbR}f\left(x\right)dx\int\limits _{a}^{b}h\left(x\right)dx\quad a.s.
\]
This proves that (\ref{eq: Ident. of limit}) holds almost surely
for $g\in G_{1}$. The proof for $g\in G_{k}$ is performed using
similar calculations coordinatewise. Thus the proof is complete.

\end{document}